\documentclass[11pt,leqno]{amsart}
\usepackage{amsmath,amsfonts,amssymb,amscd,amsthm,amsbsy,upref}
\usepackage[alphabetic]{amsrefs}
\usepackage{color}
 \usepackage{comment}
\usepackage{mathptmx} 
\usepackage{hyperref}

\newtheorem{thm}{Theorem}[section]
\newtheorem{lem}[thm]{Lemma}
\newtheorem{cor}[thm]{Corollary}
\newtheorem{prop}[thm]{Proposition}
\newtheorem{claim}[thm]{Claim}

\newtheorem{prob}[thm]{Problem}

\theoremstyle{definition}
\newtheorem{defin}[thm]{Definition}

\newtheorem{rem}[thm]{Remark}

\newtheorem{remark}[thm]{Remark}
\newtheorem{remarks}[thm]{Remarks}

\newtheorem{thmAlfa}{Theorem}

\def\E{{\mathbb E}}

\def\Q{{\mathbb Q}}
\def\N{{\mathbb N}}
\def\M{{\mathbb M}}

\def\R{{\mathbb R}}
\def\Z{{\mathbb Z}}

\def\cC{{\mathcal C}}
\def\cD{{\mathcal D}}

\def\cM{{\mathcal M}}

\def\cQ{{\mathcal Q}}

\def\cV{{\mathcal V}}

\newcommand{\kleq}{\!\leq\!}

\newcommand\kin{\!\in\!}

\newcommand{\nb}{\bar n}
\newcommand{\mb}{\bar m}

\def\vp{\varepsilon}

\newcommand{\ie}{{\it i.e.,\ }}
\newcommand{\co}{\mathrm{c}_0}
\newcommand{\diam}{\text{\rm diam}}
\newcommand{\Lip}{\text{\rm Lip}}

\newcommand{\s}{\mathrm{sum}}
\newcommand{\e}{\varepsilon}

\newcommand{\car}{\textrm{\bbold 1}}
\font\bbold=bbold12

\begin{document}

\allowdisplaybreaks

\title{Coarse and Lipschitz universality}

\author{F.~Baudier}
\address{F.~Baudier, Department of Mathematics, Texas A\&M University, College Station, TX 77843, USA}
\email{florent@math.tamu.edu}

\author{G.~Lancien}
\address{G.~Lancien, Laboratoire de Math\'ematiques de Besan\c con, Universit\'e Bourgogne Franche-Comt\'e, 16 route de Gray, 25030 Besan\c con C\'edex, Besan\c con, France}
\email{gilles.lancien@univ-fcomte.fr}

\author{P.~Motakis}
\address{P.~Motakis, Department of Mathematics, University of
Illinois at Urbana-Champaign, Urbana, IL 61801, U.S.A.}
\email{pmotakis@illinois.edu}

\author{Th.~Schlumprecht}
\address{Th.~Schlumprecht, Department of Mathematics, Texas A\&M University, College Station, TX 77843-3368, USA, and Faculty of Electrical Engineering,
Czech Technical University in Prague, Zikova 4, 16627, Prague, Czech Republic}
\email{schlump@math.tamu.edu}

\thanks{The first named author was supported by the National Science
Foundation under Grant Number DMS-1800322.
The second named author was supported by the French
``Investissements d'Avenir'' program, project ISITE-BFC (contract
 ANR-15-IDEX-03).
The third named author was  supported by the National Science Foundation
under Grant Numbers DMS-1600600 and DMS-1912897.
The fourth named author was supported by the National Science Foundation under Grant Numbers DMS-1464713 and DMS-1711076 .}

\maketitle
\begin{abstract} In this paper we  provide several  \emph{metric universality} results. We exhibit  for certain classes $\cC$ of metric spaces, families of metric spaces $(M_i, d_i)_{i\in I}$ which have the property that a  metric space $(X,d_X)$ in $\cC$ is coarsely, resp. Lipschitzly, universal for all spaces in $\cC$ if the collection of spaces $(M_i,d_i)_{i\in I}$ equi-coarsely, respectively equi-Lipschitzly, embeds into $(X,d_X)$. Such families are built as certain Schreier-type metric subsets of $\co$. We deduce
a metric analog to Bourgain's theorem, which generalized Szlenk's theorem, and prove that a space which is coarsely universal for all separable reflexive asymptotic-$c_0$ Banach spaces is coarsely universal for all separable metric spaces.
One of our coarse universality results is valid under Martin's Axiom and the negation of the Continuum Hypothesis. We discuss the strength of the universality statements that can be obtained without these additional set theoretic assumptions. 
In the second part of the paper, we study universality properties of Kalton's interlacing graphs. In particular, we prove  that every finite metric space embeds almost isometrically in some interlacing graph of large enough diameter.
\end{abstract}

 \setcounter{tocdepth}{3}
\tableofcontents
\section{Introduction}\label{S:1}
A metric space $\mathsf{Y_{cu}}$ is said to be coarsely universal for a class $\cM$ of metric spaces if every metric space in $\cM$ coarsely embeds into $\mathsf{Y_{cu}}$. By modifying the definition accordingly we can obviously consider universality in various categories: [Banach spaces$\sim$isomorphic embeddings], [metric spaces$\sim$bi-Lipschitz embeddings], etc. A natural question is thus the following: Given a class of metric spaces can we find a metric space that is universal for this class with respect to a given type of metric embedding? There are numerous embedding results that provide satisfactory answers to this broad question. That $\ell_\infty$ is isometrically universal for the class of separable metric spaces is a reformulation of the (elementary but fundamental) Fr\'echet-Kuratowski embedding theorem \cite{Frechet1910, Kuratowski}. Note that $\ell_\infty$ is not separable and thus does not belong to the class it is a universal space for. This leads us to refine the question to, say: is there a member of the class that is universal for the class itself? Urysohn's space \cite{Urysohn} answers positively this question for the class of separable metric spaces and isometric embeddings. However, it is not always possible to find a universal space within the considered class. A (relatively) simple example is the class of separable super-reflexive Banach spaces when universality refers to isomorphic embeddings. A much more difficult result of Szlenk \cite{Szlenk1968} states that there is no separable reflexive Banach space that is isomorphically universal for the class of separable reflexive Banach spaces. Szlenk's theorem was improved by Bourgain \cite{Bourgain1980} who showed that a separable Banach space that is isomorphically universal for the class of separable reflexive spaces is also isomorphically universal for \emph{all} separable Banach spaces. So if we want to show that a separable Banach space contains an isomorphic copy of every separable Banach space we only need to show that it contains an isomorphic copy of every separable reflexive Banach space. To prove this remarkable rigidity result in the context of isomorphic universality, Bourgain ingeniously incorporated techniques from descriptive set theory. Bourgain's descriptive set theoretic approach for universality problems, was further extended by Bossard \cite{Bossard2002} to show that a class of Banach spaces which is analytic, in the Effros-Borel structure of subspaces of $C[0,1]$, and contains all separable reflexive Banach spaces, must contain a universal space.

We will not discuss the numerous variants of the universality problem but instead we will focus on the following rigidity phenomenon in the context of universality. We voluntarily do not specify a specific type of embeddings.

\begin{prob}\label{P:1.1}
For what classes $\cC$ and $\cD$ of metric spaces such that $\cC\subset\cD$, a universal space for $\cC$ is also a universal space for $\cD$?
\end{prob}

The first part of the article revolves around Problem \ref{P:1.1} in the Lipschitz and coarse categories. Our first theorem says that a metric space is Lipschitzly universal for the class of all separable metric spaces, if it is universal for the uncountable collection $\cC:=\{(\mathrm{S}_\alpha(\Q),d_\infty)\colon \alpha<\omega_1\}$, which we will refer to as the collection of rational-valued smooth Schreier metric spaces. None of the metric spaces in $\cC$ is coarsely universal, but since they are built as certain Schreier-type metric subsets of $\co$, their entire hierarchy captures enough structure of $\co$, and thus confers its good universality properties.
\begin{thmAlfa}\label{T:A}
If a complete separable metric space contains  bi-Lipschitz copies of $(\mathrm{S}_\alpha(\Q),d_\infty)$ for every countable ordinal $\alpha$, then it is Lipschiztly universal for the class of all separable metric spaces.
\end{thmAlfa}

Theorem \ref{T:A} should be thought of as a purely Lipschitz analogue of the linear universality result that states that if a Banach space $X$ is isomorphically universal for the class of separable reflexive asymptotic-$\co$ Banach spaces then  $X$ contains an isomorphic copy of $\co$. This linear universality can be found in \cite{OdellSchlumprechtZsak2007}, as it is explained at the end of section \ref{S:1}. Similarly to the linear setting we use an ordinal index \`a la Bourgain.

In the context of coarse universality, technical difficulties arise and we need some additional set-theoretic axioms (Martin's Axiom and the negation of the Continuum Hypothesis) to prove a coarse analogue of Theorem \ref{T:A}. Note that here we only consider integer-valued Schreier metric spaces.

\begin{thmAlfa}\label{T:B}(MA+$\neg$CH)
If a separable metric space contains coarse copies of $(\mathrm{S}_\alpha(\Z),d_\infty)$ for every countable ordinal $\alpha$, then it is coarsely universal for the class of all separable metric spaces.
\end{thmAlfa}

We end the first part with several results which  have statements which are somewhat weaker than Theorem \ref{T:B}, but can be shown without any further axioms.  In particular, we show the following.

\begin{thmAlfa}\label{T:E}
If a separable metric space  $(M,d)$ contains coarse copies of $(\mathrm{S}_\alpha(\Z),d_\infty)$ for every countable ordinal $\alpha$,
then the class of all separable bounded metric spaces embeds equi-coarsely into $(M,d)$.
\end{thmAlfa}

With the help of a deep result of Dodos \cite{Dodos2009}, we prove  Theorem \ref{T:C} below. Note that the assumption is formally stronger than that of Theorem \ref{T:B} or Theorem \ref{T:E}.
\begin{thmAlfa}\label{T:C}
If a separable metric space is coarsely universal for the class of all reflexive asymptotic-$\co$ Banach spaces then it is coarsely universal for the class of all separable metric spaces.
\end{thmAlfa}

The second part of the article discusses some universality properties of the sequence of interlacing graphs $([\N]^k,d_{\mathrm{I}})_k$ and their applications to universality problems. The geometry of these graphs is intimately connected to the geometry of $\co$ via the summing norm, and we prove the following universality property.

\begin{thmAlfa}\label{T:D} For every finite metric space $X$ and every $\varepsilon>0$, there exists $k:=k(X,\varepsilon)\in\N$ such that $X$ admits a bi-Lipschitz embedding into $([\N]^k,d_{\mathrm{I}})$ with distortion at most $1+\varepsilon$.
\end{thmAlfa}

Note that it follows from this almost isometric universality property of the interlacing graphs and the work of Eskenazis, Mendel and Naor \cite{EskenazisMendelNaor2019} that the sequence of interlacing graphs $([\N]^k,d_{\mathrm{I}})_k$ does not equi-coarsely embed into any Alexandrov space of nonpositive curvature.

Then, we discuss the connection between metric universality, the geometry of the interlacing graphs, and a nonlinear version of Johnson-Odell elasticity.

In \cite{Kalton2007}, Kalton showed that a separable Banach $X$ that is coarsely universal for all separable metric spaces cannot have all its iterated duals separable. The argument is based on the existence of uncountably many well separated copies of the interlacing graphs in $c_0$. We conclude the paper by showing  that it can be generalized to prove the following.

\begin{thmAlfa}\label{T:F} Let  $X$ be a separable Banach space with non separable bidual $X^{**}$ and such that no spreading model generated by  a normalized  weakly null sequence in $X$  is equivalent to the $\ell_1$-unit vector basis. Assume that $X$ coarsely embeds into a Banach space $Y$. Then there exists $k\in \N$ such that $Y^{(2k)}$ is non separable.
\end{thmAlfa}
In connection with this last result, it is important to note that $\ell_1$ is known to coarsely embed into $\ell_2$.

\section{Preliminaries}\label{S:2}

\subsection{Coarse and Lipschitz geometry}
If $X$ and $Y$ are two metric spaces, the \emph{$Y$-distortion} of $X$, denoted $c_Y(X)$, is defined as the infimum of those $D\in[1,\infty)$ such that there exist $s\in(0,\infty)$ and a map $f\colon X\to Y$ so that for all $x,y\in X$
  \begin{equation}\label{E:distortion}
    s\cdot d_{X}(x,y)\leq d_{Y}\big(f(x),f(y)\big)\leq
    s\cdot D\cdot d_{X}(x,y).
  \end{equation}

When \eqref{E:distortion} holds we say that $X$ bi-Lipschitzly embeds into $Y$ with distortion $D$. We introduce some convenient terminology and notation that will allow us to treat all at once various embedding notions.

\begin{defin}
Let $X$ and $Y$ be metric spaces. Let $\rho,\omega\colon [0,\infty)\to[0,\infty)$. We say that $X$ $(\rho,\omega)$-embeds into $Y$ if there exists $f\colon X\to Y$ such that for all $x,y\in X$ we have
\begin{equation}\rho(d_X(x,y))\le d_Y(f(x),f(y))\le \omega(d_X(x,y)).
\end{equation}

If $\{X_i\}_{i\in I}$ is a collection of metric spaces. We say that $\{X_i\}_{i\in I}$ $(\rho,\omega)$-embeds into $Y$ if for every $i\in I$, $X_i$ $(\rho,\omega)$-embeds into $Y$.
\end{defin}

We will say that $\{X_i\}_{i\in I}$ equi-coarsely embeds into $Y$ if there exist non-decreasing functions $\rho,\omega\colon [0,\infty)\to[0,\infty)$ such that $\lim_{t\to\infty}\rho(t)=\infty$ and $\{X_i\}_{i\in I}$ $(\rho,\omega)$-embeds into $Y$. We  say that $\{X_i\}_{i\in I}$ equi-bi-Lipschiztly embeds into $Y$ if  $\{X_i\}_{i\in I}$ $(\rho,\omega)$-embeds into $Y$, where $\rho$ and $\omega$ are increasing and linear on $[0,\infty)$.

 Note that equi-bi-Lipschitz embeddability is a stronger condition than merely assuming that $\sup_{i\in I}c_Y(X_i)<\infty$ since it does not allow for arbitrarily large or arbitrarily small scaling factors in \eqref{E:distortion}. However if $Y$ is a Banach space rescaling is possible, and the two notions coincide.

Aharoni's embedding theorem \cite{Aharoni1974} states that there exists a universal constant $K\in[1,\infty)$ such that every separable metric space bi-Lipschitzly embeds into $\co$ with distortion at most $K$. The optimal distortion in Aharoni's embedding theorem is $K=2$ as shown in \cite{KaltonLancien2008}. A consequence of Aharoni's embedding theorem, which will be used repeatedly, is that a metric space is Lipschitzly (resp. coarsely) universal for the class of separable metric spaces if and only if it contains a bi-Lipschitz (resp. coarse) copy of $\co$.
\subsection{Trees, derivations, and Bourgain's index theory}

A tree $T$ over a set $X$ is a collection of finite sequences $(x_1,\ldots,x_n)$ of elements of a set $X$ with the property that whenever $(x_1,\ldots,x_n)$ is in $T$ then $(x_1,\ldots,x_{n-1})$ is in $T$ as well. A tree is \emph{well-founded} if it has no infinite branch, \ie there is no sequence $(x_k)_{k=1}^\infty$ in $X$ such that for all $n\in\N$ $(x_1,x_2,\dots,x_n)\in T$. There is a classical ordinal derivation on trees which is defined transfinitely as follows:

\begin{enumerate}
\item[] $T^0=T$
\item[] $T^{\alpha+1}=\{(x_1,x_2,\dots,x_n)\colon (x_1,x_2,\dots,x_n,x_{n+1})\in T^{\alpha}\}$, for any ordinal $\alpha$
\item[] $T^{\beta}=\cap_{\alpha<\beta}T^{\alpha}$ for any limit ordinal $\beta$.
\end{enumerate}
We definite $o(T)$, the order of a tree $T$, to be the least ordinal number such that $T^{o(T)}=\emptyset$, and by convention we set $o(T)=\infty$ if such an ordinal does not exist. Note that if $T$ is well-founded then the derivation produces a strictly decreasing sequence of trees and thus $o(T)<\infty$.
For every ordinal $\alpha$ it is easy to construct a tree $T_\alpha$ such that $o(T_\alpha)=\alpha$.
In Section \ref{S:2} we will need to strengthen a crucial result about trees on Polish spaces, which are complete, separable and metrizable spaces. A tree $\mathsf{T}$ on a topological space $X$ is \emph{closed} if for every $n\in\N$, $\mathsf{T}\cap X^n$ is closed in $X^n$ equipped with the product topology. The following proposition, which follows from \cite[Theorem 31.1]{Kechris1995}, was observed  by Bourgain \cite[Proposition 3]{Bourgain1980}.

\begin{prop}\label{P:2.2}
If $T$ is a closed and well founded tree on a Polish space, then $o(T)<\omega_1$, where $\omega_1$ denotes the first uncountable ordinal.
\end{prop}

In order to facilitate the reading of Section \ref{S:2}, we recall Bourgain's ordinal index ``measuring'' the presence of a given basic sequence in a Banach space. This idea was introduced in \cite{Bourgain1980} for a basis of $C[0,1]$, but can be (and has been extensively) applied for other basic sequences (see for instance Definitions 3.1 and 3.6 in \cite{AJO2005} or \cite{Odell2004}). In this article we will be mostly interested in the canonical basis of $c_0$.

Let $(e_i)_i$ be a normalized basic sequence, $X$ be a Banach space, and $K\ge 1$. Denote by $T(X,(e_i)_i, K)$ the set of finite sequences $(x_1,x_2,\dots,x_n)$ of elements in $X$ such that
\begin{equation}
\frac{1}{K}\|\sum_{k=1}^na_kx_k\|\le \|\sum_{k=1}^na_ke_k\|\le K\|\sum_{k=1}^na_kx_k\|.
\end{equation}
It is clear that $T(X,(e_i)_i, K)$ is a closed tree on $X$. It is also straightforward that $X$ contains a $K^2$-isomorphic copy of $Y=\overline{\text{span}(e_i)}$ if and only if $T(X,(e_i)_i, K)$ is not well founded (or in other words has an infinite branch). Moreover, if $X$ is separable (and thus Polish), it follows from Proposition \ref{P:2.2} that $X$ contains an $K^2$-isomorphic copy of $Y=\overline{\text{span}(e_i)}$ if and only if $o(T(X,(e_i)_i, K))=\omega_1$.
At the technical level, Bourgain constructed for every ordinal $\alpha$, a separable reflexive Banach space $X_\alpha$ such that for some universal constant $K>0$, $T(X_\alpha, (e_i)_i, K)\ge \alpha$, where $(e_i)_i$ is a basis of $C[0,1]$. If a separable Banach space $Z$ is isomorphically universal for all separable reflexive Banach spaces, it is easy to see that it must be $C$-isomorphically universal for all separable reflexive Banach spaces for some $C\ge 1$. Indeed, if there exists a sequence of reflexive separable Banach spaces $(X_n)$ so that the embedding constants of them escape to infinity, then the reflexive separable space $(\sum_nX_n)_2$ would not embed into $Z$. Thus  $Z$ will contain a $C$-isomorphic copy of all the $X_\alpha$'s and thus $T(Z, (e_i)_i, D)=\omega_1$ for some $D\ge 1$, and based on the above discussion it follows that $Z$ contains an isomorphic copy of $C[0,1]$ (which is well-known to be linearly isometrically universal for all separable Banach spaces thanks to Banach's embedding theorem \cite{Banach1932}).

Bourgain's $(e_i)$-index of $X$ is defined as follows:
$$\mathrm{I}(X, (e_i))=\sup \{o(T(X,(e_i)_i, K))\colon K\ge 1\}.$$

We collect the key properties of the Bourgain's index of the canonical basis of $c_0$, simply denoted by $I_{\co}$, that we will need later on.

\begin{prop} Let $X, Y$ be separable Banach spaces.
\begin{enumerate}
\item If $X$ is a subspace of $Y$ then $I_{\co}(X)\leq I_{\co}(Y)$.
\item If $X$ is isomorphically equivalent to $Y$ then $I_{\co}(X)=I_{\co}(Y)$.
\item $\co$ embeds isomorphically into $X$ if and only if $I_{\co}(X)\ge \omega_1$.
\end{enumerate}
\end{prop}

\subsection{Schreier sets and higher order Tsirelson spaces} Schreier sets proved to be very useful to measure indices as well as to construct Banach
spaces having certain indices. We will also use them in the more general metric context.
 We denote by $[\N]^{<\omega}$ the set of finite subsets of $\N$. An element $\nb=\{n_1,n_2,\ldots, n_k\}\in[\N]^{<\omega}$ will always be written in strictly increasing order, \ie $n_1<n_2<\ldots<n_k$. If $A$ and $B$ are finite subsets of $\N$ we write $n\le A< B$ if $n\le \min(A)\le\max(A)< \min(B)$. For a countable ordinal $\alpha$ we denote by $\mathrm{S}_\alpha\subset [\N]^{<\omega}$  the \emph{Schreier family of order $\alpha$} which is defined recursively as follows:
\begin{enumerate}
\item[] $\mathrm{S}_0=\big\{ \{n\}: n\in\N\}$
\item[] $\mathrm{S}_{\alpha+1}=\Big\{ \bigcup_{j=1}^n E_j: E_j\in \mathrm{S}_\alpha, \text{ for }j=1,2\ldots n\text{ and }n\le E_1<E_2<\ldots<E_n  \Big\}$
\item[] $\mathrm{S}_\beta=\big\{ A\in [\N]^{<\omega}: \exists n\in\N, \text{ so that }n\le A, \text{ and } A\in \mathrm{S}_{\alpha_n}\big\}$, if $\beta$ is a limit ordinal, and $(\alpha_n)\subset [0,\alpha)$ is a (fixed) sequence which increases to $\beta$.
\end{enumerate}
The above definition of $\mathrm{S}_\beta$, for $\beta$ limit ordinal, is dependent on the choice of the sequence $(\alpha_n)$, but for our purposes the specific choice of $(\alpha_n)$ will be irrelevant. The Schreier sets $(\mathrm{S}_\alpha)_{\alpha<\omega_1}$ are collections of finite subsets of $\N$ with increasing complexity which naturally generate trees $T(\mathrm{S}_\alpha):=\{(n_1,n_2,\ldots,n_k)\colon \{n_i\}_{i=1}^k\in \mathrm{S_\alpha}\}$ on $\N$. It is not difficult to prove by transfinite induction that $o(T(\mathrm{S}_\alpha))=\omega^{\alpha}+1$.

We now describe a procedure to generate metric spaces using Schreier sets. Let $\mathcal{G}$ be a family of finite subsets of $\mathbb{N}$ and let $\E$ be a non-empty (finite or infinite) countable subset of $\mathbb{R}$. We define the subset of $c_{00}(\N)$
\[X_{\mathcal{G},\E} = \Big\{\sum_{i\in G}c_ie_i: G\in\mathcal{G}, c_i\in\E\text{ for }i\in G\Big\}\]
where $(e_i)$ is the canonical basis of $c_{00}$. We will endow $X_{\mathcal{G},\E}$ with the metric $d_\infty$ induced by the standard $c_0$-norm $\|\cdot\|_\infty$. When $\mathcal{G}=\mathrm{S_\alpha}$ we will simply denote by $(\mathrm{S}_\alpha(\mathbb{E}),d_\infty)$ the metric space obtained.
These metric spaces naturally embed into the {\em higher order Tsirelson spaces} $T^*_\alpha$,  which are reflexive Banach spaces
whose duals $T_\alpha$ have norms which are  implicitly defined based on an admissibility condition that involves
the Schreier sets. Although the  original space constructed   by Tsirelson \cite{Tsirelson1974} was $T^*_\alpha$, for $\alpha=1$,
nowadays their duals $T_\alpha$, are usually referred to as {\em Tsirelson spaces}, and  it is easier to define $T^*_\alpha$ by first defining $T_\alpha$.
We recall the crucial properties of the Banach space $T_\alpha^* $ (c.f. \cite{OdellSchlumprechtZsak2007}), that are needed in this article. The separable reflexive Banach space $T_\alpha^*$ is  asymptotic-$c_0$ and has a 1-unconditional basis $(u_i)_i$ with the property that for any $G\in\mathrm{S}_\alpha$ the sequence $(u_i)_{i\in G}$ is $2$-equivalent to the unit vector basis of $\ell_\infty^{|G|}$. From the latter property it follows that the natural embedding of $(\mathrm{S}_\alpha(\mathbb{E}),d_\infty)$ in $T_\alpha^*$ (mapping $\sum_{i\in G}c_ie_i$ to $\sum_{i\in G}c_iu_i$, for $G\in \mathrm{S}_\alpha$) is a $4$-Lipschitz isomorphism.
Moreover, it follows from \cite{OdellSchlumprechtZsak2007} that Bourgain's $c_0$-index of
$T^*_\alpha$ tends to $\omega_1$ as $\alpha$ tends to $\omega_1$.

\section{Metric universality via descriptive set theory}\label{S:3}
This section is deeply inspired by the profound ideas introduced by Bourgain and Bossard in connection with isomorphic universality, and the unification of these approaches initiated by Argyros and Dodos \cite{ArgyrosDodos}.  The most natural approach to prove Theorem \ref{T:A} (resp. Theorem \ref{T:B}), is to mimic Bourgain's strategy and construct an ordinal index that will detect the presence of a bi-Lipschitz (resp. coarse) copy of $\co$, and which behaves similarly to Bourgain $\co$-index. We can indeed (though non-trivially) adjust Bourgain's approach to prove the Lipschitz universality result in Section \ref{S:2.1}. Unfortunately some difficulties arise in the coarse setting. On one hand, in Section \ref{S:2.2}, we use additional set theoretic axioms to prove Theorem \ref{T:B}. On the other hand, we need to resort to the delicate theory of strongly bounded classes of Banach spaces to prove Theorem \ref{T:C}. This is carried over in Section \ref{S:2.3} where we will use a deep theorem of Dodos. With this organization, we hope it will be clear what is the scope of application of Bourgain's strategy and why it partially fails to work in the coarse framework.

\subsection{Lipschitz universality via a Lipschitz $\co$-index}\label{S:2.1}
To detect the presence of a linear isomorphic copy of $C[0,1]$ Bourgain used a tree ordinal index where the trees are defined by a fixed basis of $C[0,1]$.  By completeness,  we only need to find a dense subset of $\co$, in order to detect a Lipschitz copy of $\co$ while to detect a coarse copy of $\co$ we only need to find a $1$-net of $\co$. Note that $X_{[\mathbb{N}]^{<\omega},\mathbb{Q}}$ is a dense subset of $c_0$ and that $X_{[\mathbb{N}]^{<\omega},\mathbb{Z}}$ is a 1-net in $c_0$. It will be very useful to understand $X_{\mathcal{G},\E}$ as the collection of all $f:\mathbb{N}\to\E$ for which there is $G\in\mathcal{G}$ so that $\mathrm{supp}(f)\subset G$.
 To handle the nonlinearity of our universality problem we will introduce combinatorial objects called vines which will be a substitute for trees. The elements of a vine $\mathcal{V}$ will also be  collections of elements of $X$, but they will be indexed over collections of finitely supported functions $f:\mathbb{N}\to\E$, where $\E$ is a fixed countable subset of $\mathbb{R}$, with $0\in \E$. Such elements will be called bunches. For a collection $\mathcal{V}$ of bunches to be called a vine it must also be closed under a certain restriction operation. Formally, for a (finite or infinite) countable subset $\E$ of $\mathbb{R}$, with $0\in\E$,  and finite subset $G$ of $\mathbb{N}$ we call the set
 \[[\E,G] = \{f:\N\to \E \text{ with }\mathrm{supp}(f)\subset G\}\]
 an \emph{$\E$-bunch}.  Note that if $G=\emptyset$, then $[\E,G]=\{0\}$, where $0:\N\to \E$ is the constant zero map.
 We put
 $$c_{00}^\E=\bigcup_{G\in[\N]^{<\omega}} [G,\E]=\big\{ (\xi_j)\subset \E: \{j\in\N: \xi_j\not=0\} \text{ is finite} \big\}$$
 which is dense in $c_0$ if $\E$ is dense in $\R$.
 Given a set $X$ and a countable subset $\E$ of $\mathbb{R}$, every element of the form $\chi = (x_f)_{f\in[\E,G]}$ in $X^{[\E,G]}$ will be called an $\E$-bunch over $X$. We define a partial order on the set of $\E$-bunches over $X$ as follows. If $\chi = (x_f)_{f\in[\E,F]}$, $\psi = (y_f)_{f\in[\E,G]}$ we shall write $\chi \preceq \psi$ if $F$ is an initial segment of $G$ and for every $f\in[\E,F]$ we have $y_f = x_f$. This makes sense because $[\E,F] \subset [\E,G]$.
 If $G=\emptyset$  then $X^{[\E,G]}$  will be in an obvious way identified with $X$, and we note that for  $G\in[\N]^{<\omega}$ and $(x_f)_{f\in[\E,G]}$, $x_0\equiv (x_f)_{f\in[\E,\emptyset]}\preceq (x_f)_{f\in[\E,G]}$, or more generally
 $(x_f)_{f\in[\E,F]}\preceq(x_f)_{f\in[\E,G]}$ for all initial segments $F$ of $G$.

A set  $\mathcal{V}$
of $\E$-bunches over $X$ is called an \emph{$\E$-vine over $X$} if for all $\chi\in\mathcal{V}$ the set $[\psi\preceq\chi]$ is a subset of $\mathcal{V}$. Note that $[\psi\preceq\chi]$ is finite and totally ordered and hence $(\mathcal{V},\preceq)$ is a tree in the abstract classical sense. We shall say that the $\E$-vine $\mathcal{V}$ is \emph{well founded} if the tree  $(\mathcal{V},\preceq)$ is well founded, \ie it contains no infinite totally ordered subsets. We define the derivatives of vines:

  For a vine $\cV$ we put
  \begin{align*}
  \mathcal{V}^{(1)} &= \mathcal{V}\setminus\{\chi\in\mathcal{V}: \chi\text{ is }\preceq\text{-maxinal}\},
   \intertext{and recursively for any ordinal }
\mathcal{V}^{(\alpha+1)} &= (\mathcal{V}^{(\alpha)})^{(1)},
\intertext{ and for a  limit ordinal $\alpha$, }
\mathcal{V}^{(\alpha)} &= \bigcap_{\beta<\alpha}\mathcal{V}^{(\beta)}.
\end{align*}
 Then, the ordinal index of $\mathcal{V}$ is $o(\mathcal{V}) = \min\{\alpha: \mathcal{V}^{(\alpha)} = \emptyset\}$. This is well defined if $\mathcal{V}$ is well founded. As for trees, under appropriate assumptions, being well founded is equivalent to having countable ordinal index. This will be proved in Proposition \ref{well founded if countable index}.

For $n\in\N\cup\{0\}$ we define
\[\mathcal{V}_{(n)} = \big\{\chi = (x_f)_{f\in[\E,G]}: |G| = n\big\} = \mathcal{V}\cap\left(\bigcup_{G\in[\mathbb{N}]^n}X^{[\E,G]}\right).\]
If $X$ is a topological space then for each $G\in[\mathbb{N}]^n$ the set $X^{[\E,G]}$ can be equipped with the product topology. Then the disjoint union $\cup_{G\in[\mathbb{N}]^n}X^{[\E,G]}$ can be endowed with the induced topology. In particular, $\mathcal{V}_{(n)}$ is a topological space. We shall call $\mathcal{V}$ a \emph{closed $\E$-vine} if $\mathcal{V}_{(n)}$ is a closed subset of $\cup_{G\in[\mathbb{N}]^n}X^{[\E,G]}$ for all $n\in\mathbb{N}$. This is equivalent to saying that for all $G\in[\mathbb{N}]^{<\omega}$ the set $\mathcal{V}\cap X^{[\E,G]}$ is closed. Note that $\mathcal{V}$ being closed does \emph{not} imply that the set $\cup_{\chi = (x_f)_{f\in[\E,G]}\in\mathcal{V}}\{x_f: f\in [\E,G]\}$ is a closed subset of $X$.

We can define $\pi_n:\mathcal{V}_{(n+1)}\to\mathcal{V}_{(n)}$ as follows. If $G\in[\mathbb{N}]^{n+1}$ set $G' = G\setminus\{\max(G)\}$. Given $\chi = (x_f)_{f\in[\E,G]}$ in $\mathcal{V}_{(n+1)}$ we define $\pi_n(\chi) = (x_f)_{f\in[\E,G']}$, which is in $\mathcal{V}_{(n)}$. Note that a collection $\mathcal{V}$ of $\E$-bunches over $X$ is an $\E$-vine if and only if for all $n\in\mathbb{N}$ we have that $\pi_n[\mathcal{V}_{(n+1)}]\subset\mathcal{V}_{(n)}$. Also, if $X$ is a topological space then $\pi_n$ is a continuous function.

The following is an analogue for vines of \cite[Lemma 2]{Bourgain1980} and the proof is nearly identical.

\begin{lem}
\label{no approximate maximal branches}
Let $\E$ be a countable subset of $\mathbb{R}$ and $\mathcal{V}$ be a closed $\E$-vine over a complete metric space $(X,d)$. Assume that for all $n\in\mathbb{N}$ we have $\mathcal{V}_{(n)} = \overline{\pi_n[\mathcal{V}_{(n+1)}}]$. Then either $\mathcal{V} = \emptyset$ or $\mathcal{V}$ is not well founded.
\end{lem}

\begin{proof}
We fix an enumeration $\{\epsilon_i:i\in\mathbb{N}\}$ of $\E$ and for all $n\in\mathbb{N}$ we set $\E_n = \{\epsilon_1,\ldots,\epsilon_n\}$. Assuming $\mathcal{V}\neq \varnothing$, we can find an $x_0\equiv (x_f)_{f\in[\E,\emptyset]}\in \cV\cap X=\cV_{(0)}$.
Since $V_{(0)}=\overline{\pi_0(V_{(1)})}$
we find  $\chi_1 = (x^{(1)}_f)_{f\in[\E,\{k_1\}]}\in\mathcal{V}_{(1)}$ so that $\|\pi_0(\chi_1)-x_0\|<1$. By assumption there exists $k_2>k_1$ and $\chi_2 = (x^{(2)}_f)_{f\in[\E,\{k_1,k_2\}]}\in\mathcal{V}_{(2)}$ so that for $f\in[\E_1,\{k_1\}]$ we have $d(x^{(1)}_f,x^{(2)}_f) \leq 1/2$. Proceed inductively to find an increasing sequence of integers $(k_m)_{m=1}^\infty$ and a sequence $(\chi_m)_{m=1}^\infty$ so that $\chi_m = (x^{(m)}_f)_{f\in[\E,\{k_1,k_2,\ldots,k_m\}]}\in\mathcal{V}_{(m)}$ and for all $m\in\mathbb{N}$ and $f\in[\E_m,\{k_1,\ldots,k_m\}]$ we have $d(x^{(m)}_f,x_f^{(m+1)}) \leq 1/2^m$. We conclude that for any $m_0\in\mathbb{N}$ and $f\in[\E_{m_0},\{k_1,\ldots,k_{m_0}\}]$ the sequence $(x_f^{(m)})_{m\geq m_0}$ is Cauchy and we denote its limit by $y_f$. Because $\mathcal{V}$ is an $\E$-vine it is closed under taking projections $\pi_n$ and because $\mathcal{V}$ is assumed to be closed we deduce that $\psi_m = (y_f)_{f\in[\E,\{k_1,\ldots,k_m\}]}$ is in $\mathcal{V}$ for all $m\in\mathbb{N}$. Because $(\psi_m)_m$ is an infinite chain the $\E$-vine $\mathcal{V}$ must be ill founded.
\end{proof}

The following is the analogue of Proposition \ref{P:2.2} for vines.

\begin{prop}
\label{well founded if countable index}
Let $\E$ be a countable subset of $\mathbb{R}$ and $\mathcal{V}$ be a closed $\E$-vine on a Polish space. If $\mathcal{V}$ is well founded then $o(\mathcal{V})<\omega_1$.
\end{prop}

\begin{proof}
We will show that there is $\eta<\omega_1$ so that $\mathcal{V}^{(\eta)} = \emptyset$.  It is easily observed that for any $n\in\mathbb{N}$ and ordinal $\alpha$ we have
\begin{equation}\label{stating the obvious}(\mathcal{V}^{(\alpha+1)})_{(n)} = \pi_n[(\mathcal{V}^{(\alpha)})_{(n+1)}],
\end{equation}
\ie a $\chi$ of length $n$ is in $\mathcal{V}^{(\alpha+1)}$ if an only if it is the direct predecessor of a $\psi$ of length $n+1$ in $\mathcal{V}^{(\alpha)}$.
For $n\in\mathbb{N}$, consider the decreasing hierarchy of closed sets $\overline{(\mathcal{V}^{(\alpha)})_{(n)}}$, $\alpha<\omega_1$ of $\cup_{G\in[\mathbb{N}]^n}X^{[\E,G]}$. Because $X$ is Polish, so is $\cup_{G\in[\mathbb{N}]^n}X^{[\E,G]}$ and therefore there must exist an $\alpha_n <\omega_1$ so that for all $\beta>\alpha_n$ we have $\overline{(\mathcal{V}^{(\alpha_n)})_{(n)}} = \overline{(\mathcal{V}^{(\beta)})_{(n)}}$. This is because in a Polish space there can be no strictly increasing transfinite hierarchy of open sets of length $\omega_1$. Take $\eta = \sup_n\alpha_n$ and define $\mathcal{W} = \cup_{n=1}^\infty \overline{(\mathcal{V}^{(\eta)})_{(n)}}$. We observe  that $\mathcal{W}$ is an $\E$-vine over $X$. We show that $\mathcal{W}$ satisfies the assumption of Lemma \ref{no approximate maximal branches}. Indeed, for $n\in\mathbb{N}$ we have
\begin{align*}
\mathcal{W}_{(n)} &= \overline{(\mathcal{V}^{(\eta)})_{(n)}} = \overline{(\mathcal{V}^{(\eta+1)})_{(n)}} \text{ (by the choice of $\eta$)}\\
&=\overline{\pi_n[(\mathcal{V}^{(\eta)})_{(n+1)}]} \text{ (by \eqref{stating the obvious})}\\
&= \overline{\pi_n\Big[\overline{(\mathcal{V}^{(\eta)})_{(n+1)}}\Big]} \text{ (by continuity of $\pi_n$)}\\
&=\overline{\pi_n[\mathcal{W}_{(n+1)}]}.
\end{align*}
This means that either $\mathcal{W}=\emptyset$ or $\mathcal{W}$ is ill founded. Because $\mathcal{V}$ is closed $\mathcal{W}\subset\mathcal{V}$ and because $\mathcal V$ is well founded, so is $\mathcal W$ and hence $\mathcal{W} = \emptyset$. It follows that $\mathcal{V}^{(\eta)} = \cup_{n\in\mathbb{N}}(\mathcal{V}^{(\eta)})_{(n)}\subset\mathcal{W} = \emptyset$. Therefore, $o(\mathcal{V})\leq \eta$.
\end{proof}

We can now introduce an ordinal index that will capture the presence of a bi-Lipschitz copy of $\co$ in a metric space.
For any $C>0$, any metric space $(M,d)$, and any countable subset $\mathbb{E}$ of $\R$, it is easy to verify that the set (think of $[\E,G]$ being a subset of $c_0$)
\[
\begin{split}
\mathcal{V}(M,\mathbb{E}, C) = \left\{(x_f)_{f\in [\mathbb{E},G]}:\,\begin{matrix} G\in[\mathbb{N}]^{<\omega}, x_f\in M \text{ for }f\in[\mathbb{E},G],\text{ and}\\
\forall f,g\kin[\mathbb{E},G]\quad\frac{1}{C}\|f-g\|_\infty \leq d(x_f,x_g)\leq C\|f-g\|_\infty\end{matrix}\right\},
\end{split}
\]
is a closed $\mathbb{E}$-vine on $M$.
We define the Lipschitz $\co$-index of $M$ as
$$\mathrm{I}^{\Lip}_{\co}(M)=\sup \{o(\mathcal{V}(M,\mathbb{Q}, C)\colon C>0\}.$$

\begin{prop}\label{Lipschitzindex}
Let $M$ be a Polish space. Then,
\[\co \text{ bi-Lipschitzly embeds into } M \text{ if and only if } \mathrm{I}^{\Lip}_{\co}(M)\ge \omega_1.\]
\end{prop}

\begin{proof}
The necessary implication is easy. Indeed, if $\psi$ is a Lipschitz embedding from $c_0$ into $M$, define for $G\in [\N]^{<\omega}$ and $f\in [\Q,G]$, $x_f=\psi(\sum_{i\in G} f(i)e_i)$. Then, for some $C\ge 1$, the set $\{(x_f)_{f\in [\mathbb{Q},G]}:\; G\in[\mathbb{N}]^{<\omega}\}$ is included in $\mathcal{V}(M,\mathbb{Q}, C)$ which is therefore ill founded.

Assume now that $\mathrm{I}^{\Lip}_{\co}(M)=\omega_1$, then for every countable ordinal $\alpha$ there exist $C_\alpha>0$ such that $o(\mathcal{V}(M,\mathbb{Q}, C_\alpha)\ge \alpha$.
Using a simple pigeonhole argument we can find $C\ge 1$ and an uncountable sub-collection $U$ of $[1,\omega_1)$, such that for all $\alpha \in U$ we have $C_\alpha\le C$. Since obviously $o(\mathcal{V}(M,\mathbb{Q}, C))\ge o(\mathcal{V}(M,\mathbb{Q}, C_\alpha))\ge \alpha$ for every $\alpha \in U$, it follows from Proposition \ref{well founded if countable index} that $\mathcal{V}(M,\mathbb{Q}, C)$ is not well founded, \ie there exists a strictly increasing sequence of integers $(k_m)_m$
 and for $m\in\N\cup\{0\}$ an $M$-bunch $\chi_m=\big(x^{(m)}_f: f\in [\{k_1,k_2,\ldots,k_m\}, \E]\big)\in \mathcal{V}(M,\mathbb{Q}, C)$ so that $\chi_0\preceq\chi_1\preceq\chi_2\ldots$.
 But this means that  for every  finitely supported $f:\{k_1,k_2,k_3,\ldots, \}\to \Q$ there is an $x_f\in M$, so that $\chi_m=\big(x_f: f\in [\{k_1,k_2,\ldots,k_m\}, \E]\big)$, for $m\in\N$.
We define
\begin{align*}
&\psi: c_{00}^{\Q}\to M\ \text{ by } \psi((q_j)_j)=x_{f}\ \text{ where } f: \{k_1,k_2,\ldots \}\to \Q,\ \text{ is defined by } f(k_j)=q_j.
\end{align*}
It follows that $\psi $ is a bi-Lipschitz embedding  from $c_{00}^{\Q}$ (with the $c_0$-norm) into $M$. Since  $c_{00}^{\Q}$ is dense in $c_0$ and $M$ is complete, $\psi$ can be extended to a bi-Lipschitz embedding from $c_0$ into $M$.
\end{proof}

To complete the proof of Theorem \ref{T:A} it remains to show that if a complete separable metric space $M$ is Lipschitz-universal for the collection of rational valued Schreier metrics then $\mathrm{I}^{\Lip}_{\co}(M)\ge \omega_1$.

\begin{proof}[Proof of Theorem \ref{T:A}]
Assume that  for every ordinal $\alpha$, $(M,d)$ admits bi-Lipschitz embeddings of $(\mathrm{S}_\alpha(\Q),d_\infty)$. Thus, after an eventual extraction argument, there exist a constant $C>0$, an uncountable $A\subset[0,\omega_1)$, and maps $F_\alpha\colon (\mathrm{S}_\alpha(\Q),d_\infty)\to (M,d)$, $\alpha\in A$, such that for all $f,g\in \mathrm{S}_\alpha(\Q)$ and $\alpha\in A$
\begin{equation}\label{E:3}
\frac{1}{C}\|f-g\|_\infty \leq d(F_{\alpha}(f),F_{\alpha}(g))\leq C\|f-g\|_\infty.
\end{equation}
It follows that $\mathcal{V}(M,\mathbb{Q}, C)$ has ordinal index at least $o(\mathrm{S}_\alpha) = \omega^\alpha+1$, for all $\alpha\in A$. To see this, define for every $f$ in $\mathrm{S}_\alpha(\Q)$ the vector  $x_f = F_{\alpha}(f)$ and let $\mathcal{W} = \{(x_f)_{f\in[\mathbb{Q},G]}: G\in\mathrm{S}_\alpha\}$, which is thanks to \eqref{E:3}  a sub-vine of $\mathcal{V}(M,\mathbb{Q}, C)$ that has the same tree index as $\mathrm{S}_\alpha$.
\end{proof}

\subsection{Coarse universality via a coarse $\co$-index in MA+$\neg$CH}\label{S:2.2}
The technique from Section \ref{S:2.1} do not seem to be robust enough to prove  the statement of  Theorem \ref{T:B} without any further set theoretic assumptions. The main roadblock is that the simple extraction argument that provides equi-bi-Lipschitz embeddings from an uncountable collection of bi-Lipschitz embeddings does not hold in the coarse setting. Under some additional set-theoretic axioms, MA+$\neg$CH, we can prove Theorem \ref{T:B}. The advantage of assuming that Martin's Axiom holds, but the Continuum Hypothesis fails, lies in the fact that the following {\em diagonalization property} of infinite subsets of $\N$ (cf. \cite[page 3ff]{Fremlin1984} will be valid.
\begin{lem}(MA+$\neg$CH)\label{L:}
Let $(N_\alpha)_{\alpha<\omega_1}\subset [\N]^\omega$ have the property that $N_\beta\setminus N_\alpha$ is finite whenever $\alpha<\beta$ (in which case we say that $N_\beta $ is {\em almost contained } in $N_\alpha$ and write $N_\beta\subset^a N_\alpha$).
Then there exists $N$ in $[\N]^\omega$ so that $N\subset^a N_\alpha$, for all $\alpha<\omega_1$.
\end{lem}

This diagonalization property will now be used to prove an ``equi-regularization'' principle for expansion and compression moduli. Let us detail the case of the compression modulus. We first need some preparation. Denote $\mathcal I$ the class of all non decreasing maps $f:\N \to \N\cup \{0\}$ satisfying $f(1)=0$, $\lim_{n\to \infty}f(n)=\infty$ and $f(n+1)\le f(n)+1$ for all $n\in \N$. It will be useful to note that the map $j:\mathcal I \to [\N]^\omega$, defined by $j(f)=\{n\in \N,\ f(n+1)>f(n)\}$ is a bijection, and that its inverse is given by $j^{-1}(A)(n)=\sum_{i<n}\car_A(i)$, for $A\in [\N]^\omega$ and $n\in \N$. We shall also use the following easy fact. If $j(f)=\{m_1<m_2<\cdots\}$ and $j(g)=\{n_1<n_2<\cdots\}$ with $n_i\le m_i$ for all $i\in \N$, then $f\le g$. In particular, if $j(f)\subset j(g)$, then $f\le g$. We start with an easy lemma.

\begin{lem}\label{monotone_rho} Let $(g_\alpha)_{\alpha<\omega_1} \subset \mathcal I$. Then there exists $(f_\alpha)_{\alpha<\omega_1} \subset \mathcal I$ such that
\begin{enumerate}
    \item For all $\alpha<\omega_1$ and all $n\in \N$, $f_\alpha(n)\le g_\alpha(n)$.
    \item For all $\alpha<\beta<\omega_1$, $j(f_\beta)\subset^a j(f_\alpha)$.
\end{enumerate}
\end{lem}

\begin{proof} We shall build $(f_\alpha)_{\alpha<\omega_1}$ by transfinite induction. So, set $f_1=g_1$ and assume that $\beta_0<\omega_1$ is such that we have found $(f_\alpha)_{\alpha<\beta_0}$ satisfying (1) and (2). Since $\{\alpha<\beta_0\}$ is countable, a classical diagonal argument yields the existence of $M\in [\N]^\omega$ such that $M\subset^a j(f_\alpha)$ for all $\alpha<\beta_0$. Let $j(g_{\beta_0})=\{n_1<n_2<\cdots\}$. Then pick $m_1<m_2<\cdots \in M$ so that $n_i\le m_i$ for all $i\in \N$ and set $f_{\beta_0}=j^{-1}(\{m_1,m_2,\ldots\})$. We have that $f_{\beta_0}\le g_{\beta_0}$ and $j(f_{\beta_0})\subset M\subset^a j(f_\alpha)$ for all $\alpha<\beta_0$. This concludes our induction.
\end{proof}

Armed with Lemma \ref{L:} we can now prove our ``equi-regularization" principle below for compression moduli.

\begin{prop}\label{equirho}(MA+$\neg$CH) Let $(\rho_\alpha)_{\alpha<\omega_1}$ be a family of non decreasing maps from $[0,\infty)$ to $[0,\infty)$ and so that $\lim_{t\to \infty}\rho_\alpha(t)=\infty$ for all $\alpha<\omega_1$. Then there exist an uncountable subset $C$ of $\omega_1$ and $\rho:[0,\infty) \to [0,\infty)$ such that $\rho \le \rho_\alpha$ for all $\alpha \in C$ and $\lim_{t\to \infty}\rho(t)=\infty$.
\end{prop}

\begin{proof} First, note that for all $\alpha<\omega_1$, we can find $g_\alpha \in \mathcal I$ such that $g_\alpha(n)\le \rho_\alpha(n)$, for all $n\in \N$. Then consider the family $(f_\alpha)_{\alpha<\omega_1}$ associated to $(g_\alpha)_{\alpha<\omega_1}$ through Lemma \ref{monotone_rho}. Next, we apply Lemma \ref{L:} to get $M\in [\N]^\omega$ such that $M\subset^a j(f_\alpha)$ for all $\alpha <\omega_1$. For $n\in \N$, denote
$$C_n=\big\{\alpha<\omega_1,\ M\cap\{n,n+1,\ldots\} \subset j(f_\alpha)\big\}.$$
Clearly, there exists $n_0\in \N$ such that $C_{n_0}$ is uncountable. We set $C=C_{n_0}$ and define $f=j^{-1}(M\cap \{n_0,n_0+1,\ldots\})\in \mathcal I$. Then, for all $\alpha \in C$, we have $j(f)\subset j(f_\alpha)$ and therefore $f\le f_\alpha$. Finally,  $\rho$ defined by $\rho=0$ on $[0,1)$ and $\rho=f(n)$ on $[n,n+1)$, for $n\in \N$, is the desired map.
\end{proof}

Similarly, for expansion moduli, we have.

\begin{prop}\label{equiomega}(MA+$\neg$CH) Let $(\omega_\alpha)_{\alpha<\omega_1}$ be a family of non decreasing maps from $[0,\infty)$ to $[0,\infty)$ and so that $\omega_\alpha(0)=0$. Then there exist an uncountable subset $C$ of $\omega_1$ and $\omega:[0,\infty) \to [0,\infty)$ such that $\omega \ge \omega_\alpha$ for all $\alpha \in C$.
\end{prop}

\begin{proof}  The argument is very similar. Let us just describe the few adjustments. We now consider the class $\mathcal J$ of all functions $f:\N\cup\{0\} \to \N\cup\{0\}$ such that $f(0)=0$ and $f(n+1)\ge f(n)+1$ for all $n\ge 0$. The map $k:\mathcal J \to [\N]^\omega$ defined by $k(f)=k(\N)$ is a bijection.
Then, for every $\alpha<\omega_1$, there exists $g_\alpha \in \mathcal J$ such that $g_\alpha(n)\ge \omega_\alpha(n)$ for all $n\in \N$. Playing the same game as before, but with the sets $k(g_\alpha)$ instead of $j(g_\alpha)$, we obtain (under MA+$\neg$CH) the existence of an uncountable subset $C$ of $\omega_1$ and of $g\in \mathcal J$ such that $g\ge g_\alpha$ for all $\alpha \in C$. The proof is then concluded by setting $\omega(0)=0$ and $\omega=g(n)$ on $(n-1,n]$, for $n\in \N$.
\end{proof}

From Propositions \ref{equirho} and \ref{equiomega}, we deduce immediately.

\begin{prop}(MA+$\neg$CH)\label{equicoarse}
If $(X_\alpha, d_\alpha)_{\alpha<\omega_1}$ is a collection of metric spaces such that for all $\alpha<\omega_1$, $X_\alpha$ embeds coarsely into a metric space $(M,d)$, then there exists an uncountable subset $C$ of $\omega_1$ such that $(X_\alpha, d_\alpha)_{\alpha\in C}$ embeds equi-coarsely into $(M,d)$.
\end{prop}

\begin{proof}[Proof of Theorem \ref{T:B}] The argument goes along essentially the same lines as the proof of Theorem \ref{T:A}, modulo the fact that we have to work with vines defined in terms of the compression and expansion moduli. Let us outline the main steps and the place where (MA+$\neg$CH) is used.

Let $\rho,\omega$ be two elements of the class $\mathcal F$ of all non decreasing functions from $[0,\infty)\to [0,\infty)$ that are vanishing at $0$ and tending to $\infty$ at $\infty$. Let also $(M,d)$ be a complete separable metric space. Then, we define
\[
\begin{split}
\mathcal{V}(M,\mathbb{Z},\rho,\omega) = \Big\{(x_f)_{f\in [\mathbb{Z},G]}:&\; G\in[\mathbb{N}]^{<\omega}, x_f\in M \text{ for }f\in[\mathbb{E},G],\text{ and for }f,g\in[\mathbb{E},G]\\
&\text{ we have }\rho(\|f-g\|_\infty) \leq d(x_f,x_g)\leq \omega(\|f-g\|_\infty)\Big\},
\end{split}
\]
and the coarse $c_0$-index of $M$ as
$$\mathrm{I}^{\text{coarse}}_{\co}(M)=\sup \big\{o(\mathcal{V}(M,\mathbb{Z},\rho,\omega)\colon \rho,\omega \in \mathcal F\big\}.$$

The next step is to prove the analogue of Proposition \ref{Lipschitzindex}: $c_0$ coarsely embeds into $M$ if and only if $\mathrm{I}^{\text{coarse}}_{\co}(M)\ge \omega_1$. For the non trivial implication, the pigeonhole argument yielding a uniform constant $C$ is replaced by Propositions \ref{equirho} and \ref{equiomega} to prove the existence of $\rho, \omega \in \mathcal F$ such that $\mathcal{V}(M,\mathbb{Z},\rho,\omega)$ is not well founded (this is where (MA+$\neg$CH) is used). Then it implies the existence of a coarse embedding of the integer grid of $c_0$ (and therefore of $c_0$) into $M$.

Finally, assume that a separable metric space $(M,d)$, that we may assume to be complete, contains a coarse copy of all spaces $(\mathrm{S}_\alpha(\Z),d_\infty)$, for $\alpha <\omega_1$. As in the proof of Theorem \ref{T:A}, this implies that $\mathrm{I}^{\text{coarse}}_{\co}(M)\ge \omega_1$, which finishes the proof.
\end{proof}

\begin{remark} We recall that a metric space $X$ \emph{coarse-Lipschitz} embeds into a metric space $Y$ if $X$  $(\rho,\omega)$-embeds into $Y$ where, for all $t\ge 0$, $\rho(t)=At-B$ and $\omega(t)=Ct+D$ for some constants $A,B,C,D>0$. It follows clearly from the tools and arguments developed in the last two subsections that we have,  without assuming any further set theoretical  axioms, the following statement: a separable metric space containing coarse-Lipschitz all the spaces $(\mathrm{S}_\alpha(\Z),d_\infty)$, for $\alpha<\omega_1$, must a contain a coarse-Lipschitz copy of $c_0$.
\end{remark}

It is natural to wonder if Theorem \ref{T:B} holds without MA+$\neg$CH.

\begin{prob}
If a separable metric space contains coarse copies of $(\mathrm{S}_\alpha(\Z),d_\infty)$ for every countable ordinal $\alpha$, is it coarsely universal for the class of all separable metric spaces?
\end{prob}

We discuss some positive partial results in Section \ref{sec:4}.

\subsection{Coarse universality via strong boundedness}\label{S:2.3}
While we do not know how to prove Theorem \ref{T:B}  without further set axioms, we can prove Theorem \ref{T:C}. Recall that the canonical embedding of $(\mathrm{S}_\alpha(\mathbb{Z}),d_\infty)$ in $T_\alpha^*$ is a $4$-Lipschitz isomorphism onto its image, and thus the stronger assumption that the metric space contains every separable reflexive asymptotic-$\co$ space, rather than merely the collection of metric spaces $(\mathrm{S}_\alpha(\Z))_{\alpha<\omega_1}$, allows us to take advantage of the deep theory of strongly bounded classes of Banach spaces introduced by Argyros and Dodos \cite{ArgyrosDodos}. A class $\mathcal{C}$ of separable Banach spaces is said to be \emph{strongly
bounded} if for every analytic subset $A$ of $\mathcal{C}$, there exists $Y\in\mathcal{C}$ that contains isomorphic copies of every $X\in A$. Recall also that an infinite-dimensional Banach space $X$ is said to be \emph{minimal} if $X$ isomorphically embeds into every infinite-dimensional subspace of itself (e.g. the classical sequence space $\co$ is minimal). We will need the following deep result of Dodos \cite[Theorem 7]{Dodos2009}.

\begin{thm}\label{T:3.8}
For any infinite-dimensional minimal Banach space $Z$ not containing $\ell_1$, the class
$$\mathsf{NC}_Z:=\{Y\in \mathsf{SB}\colon Z \text{ does not linearly embed into  }Y\}$$ is strongly bounded.
\end{thm}

\begin{proof}[Proof of Theorem \ref{T:C}] Denote $\mathsf{R}$ the set of all reflexive elements of  $\mathsf{SB}$ and $\mathsf{As}_{c_0}$ the set of all  elements of $\mathsf{SB}$ that are asymptotic-$c_0$.
Let now $M$ be a separable metric space such that every space in
$\mathsf{R}\cap \mathsf{As}_{c_0}$ coarsely embeds into $M$. If we denote
$\mathsf{CE}_M=\{Y\in\mathsf{SB}:Y$ coarsely embeds into $M\}$, we have that $\mathsf{R}\cap \mathsf{As}_{c_0}\subset \mathsf{CE}_M$. It is easily
checked that $\mathsf{CE}_M$ is analytic (see the proof of Theorem
1.7 - section 7.1 in \cite{Braga2019}). Recall that we denoted $\mathsf{NC}_{\co}$ the set of all $Y\in \mathsf{SB}$ such that $c_0$ does not linearly embed into $Y$. If we assume, aiming for a contradiction that $\mathsf{CE}_M \subset \mathsf{NC}_{\co}$, since $\mathsf{CE}_M$ is an analytic subset of $\mathsf{NC}_{\co}$, which is strongly bounded by Theorem \ref{T:3.8}, there would exist $X \in \mathsf{NC}_{\co}$ such that any element of $\mathsf{CE}_M$, and therefore any element of $\mathsf{R}\cap \mathsf{As}_{c_0}$, linearly embeds into $X$. This is actually impossible since Bourgain's $\co$-index of the separable, reflexive and asymptotic-$c_0$ space $T^*_\alpha$ tends to $\omega_1$ as $\alpha$ tends to $\omega_1$ (see \cite{OdellSchlumprechtZsak2007}). Therefore Bourgain's $\co$-index of $X$ would be uncountable and $X$ would contain an isomorphic copy of $\co$; a contradiction with $X\in \mathsf{NC}_{\co}$. So we can now deduce the existence of $Y\in \mathsf{CE_M}$ such that $c_0$ linearly embeds into $Y$, and hence by composition $\co$ coarsely embeds into $M$. Since by a theorem of Aharoni \cite{Aharoni1974}, every separable metric space bi-Lipschitzly embeds into $c_0$, every separable metric space coarsely embeds into $M$. This concludes our proof.
\end{proof}

\begin{remarks} The same technique was used by B. de Mendon\c{c}a Braga in \cite{Braga2019} to prove that a Banach space which is coarsely universal for all reflexive separable Banach spaces is coarsely universal for all separable metric spaces.

The reader will easily adapt the above proof to show that a Banach  space that is Lipschitz universal for $\mathsf{R}\cap \mathsf{As}_{c_0}$ is Lipschitz universal for all separable metric spaces. But this was also a consequence of Theorem \ref{T:A}.

\end{remarks}

\section{Coarse universality and barycentric gluing}
\label{sec:4}
The motivation for this section is to provide a somewhat weaker statement than Theorem $\ref{T:B}$,  that does not require MA+$\neg$CH. We will show that containing coarse copies of the Schreier metric spaces $(\mathrm{S}_\alpha(\Z),d_\infty)$ for every $\alpha<\omega_1$, is a sufficient condition to equi-coarsely contain every separable \emph{bounded} metric spaces. The reason we have the boundedness restriction is because  without MA+$\neg$CH we only have the following equi-regularization principle (which is weaker than the equi-regularization principle obtained under MA+$\neg$CH).

\begin{lem}
\label{partially equi-coarse}
Let $C_0$ be an uncountable subset of $\omega_1$. Assume that  for each $\alpha\in C_0$, we have increasing functions $\rho_\alpha,\omega_\alpha:[0,\infty)\to[0,\infty)$ so that $\rho_\alpha(t)\leq \omega_\alpha(t)$ for all $t\in[0,\infty)$ and $\lim_{t\to\infty}\rho_\alpha(t)=\infty$. Then there exist increasing functions $\rho,\omega:[0,\infty)\to[0,\infty)$ and a decreasing nested sequence $(C_k)_{k\in\N}$ of uncountable subsets of $\omega_1$ so that
\begin{itemize}

\item[(i)] $\rho(t)\kleq \omega(t)$ for all $t\kin[0,\infty)$, and  $\lim_{t\to\infty}\rho(t)=\infty$,

\item[(ii)] for all $k\in\mathbb{N}$, $\alpha\in C_k$, and $0\leq t\leq k$ we have $\rho(t) \leq \rho_\alpha(t)$ and $\omega_\alpha(t)\leq \omega(t)$.
\end{itemize}
\end{lem}

\begin{proof}
For all $k\in\mathbb{N}$ and $\alpha\in C_0$ define
\[s(\alpha,k) = \min\{t\in\mathbb{N}: \rho_\alpha(t)\geq k\},\quad t(\alpha,k) = \max\{t\in\mathbb{N}: \omega_\alpha(t)\leq k\}\]
and also define
\[M(\alpha,k) = \min\{n\in\N: s(\alpha,k)<s(\alpha,n)\},\quad N(\alpha,k) = \min\{n\in\N: t(\alpha,k)<t(\alpha,n)\}.\]

Because, for each fixed $k\in\mathbb{N}$, the sets $\{s(\alpha,k):\alpha\in C_0\}$, $\{t(\alpha,k):\alpha\in C_0\}$,  $\{M(\alpha,k):\alpha\in C_0\}$, $\{N(\alpha,k):\alpha\in C_0\}$ are all countable we may find uncountable sets $C_0\supset C_1\supset C_2\supset\cdots$ so that for each $k\in\mathbb{N}$ and $\alpha,\beta\in C_k$ we have $s(\alpha,k) = s(\beta,k) = s_k$, $t(\alpha,k) = t(\beta,k) = t_k$, $M(\alpha,k) = M(\beta,k) = M_k$, and $N(\alpha,k) = N(\beta_k) = N_k$. Clearly, we have  $s_k\leq s_{k+1}$ and $t_k\leq t_{k+1}$,  for all $k\in\N$. We also observe that $\lim_ks_k = \lim_kt_k = \infty$. Indeed, it is easy to see that $s_k < s_{M_k}$ and $t_k < t_{N_k}$.  Pick $k_1<k_2<\cdots$ so that $(s_{k_j})_j$ and $(t_{k_j})_j$ are both strictly increasing.

We now define $\rho,\tilde\omega:[0,\infty)\to[0,\infty)$ as follows.
\[
\rho(t) =
\left\{
	\begin{array}{ll}
		0  & \mbox{if } 0\leq t <s_{k_1}\\
		k_j & \mbox{if } s_{k_j}\leq t < s_{k_{j+1}}, j\in\mathbb{N}
	\end{array}
\right.
,
\tilde\omega(t) =
\left\{
	\begin{array}{ll}
		k_1  & \mbox{if } 0\leq t \leq t_{k_1}\\
		k_j & \mbox{if } t_{k_{j-1}} < t \leq t_{k_{j}}, j\geq 2
	\end{array}
\right.
\]
and $\omega(t) = \rho(t)\vee\tilde\omega(t)$. The conclusion follows straightforwardly after observing that $s_{k_j},t_{k_j} \geq j$ for all $j\in\mathbb{N}$.
\end{proof}

Using the concept of vines introduced in Section \ref{S:2.2} in the coarse context we now deduce the following.

\begin{thm}
\label{T:4.2}
Let $(M,d)$ be a separable metric space and assume that for every $\alpha<\omega_1$ the metric space $(\mathrm{S}_\alpha(\mathbb{Z}),d_\infty)$ embeds coarsely into $(M,d)$. Then the class of all separable bounded metric spaces embeds equi-coarsely into $(M,d)$.

\noindent More precisely, there exist $m_0\in M$ and equi-coarse embeddings $F_n: B_n = \{x\in c_0: \|x\|_\infty\leq n\}\to (M,d)$ so that for all $n\in\mathbb{N}$ we have $F_n(0) = m_0$.
\end{thm}

\begin{proof}
We will first find $m_0\in M$ and $\rho,\omega:[0,+\infty)\to[0,+\infty)$ tending to $\infty$ at $\infty$, so that for any $\alpha<\omega_1$ and $n\in\mathbb{N}$ there exists $F_{\alpha,n}:B_{\alpha,n} = \{x\in X_{\mathrm{S}_\alpha,\mathbb{Z}}: \|x\|_\infty\leq n\}\to M$ with $F_{\alpha,n}(0) = m_0$ and for all $x,y\in B_{\alpha,n}$ we have $\rho(\|x-y\|_\infty) \leq d(F_{\alpha,n}(x),F_{\alpha,n}(y)) \leq \omega(\|x-y\|_\infty)$.
Let  $\tilde M$  be a countable  dense subset $M$, since $(\mathrm{S}_\alpha(\mathbb{Z}),d_\infty)$ is a uniformly discrete metric space, it follows from a straightforward perturbation argument
that every $(\mathrm{S}_\alpha(\mathbb{Z}),d_\infty)$ embeds coarsely into $(\tilde M,d)$ via a map $f_\alpha$ with compression and expansion moduli $\rho_\alpha,\omega_\alpha$. By passing to an uncountable set $C_0\subset \omega_1$ we can assume that there exists $m_0\in \tilde M$ so that for all $\alpha\in C_0$ we have $f_\alpha(0) = m_0$.

Fix $n\in\mathbb{N}$. Take the functions $\rho,\omega$ and the sets $C_0\supset C_1\supset C_2\supset\cdots$ given by Lemma \ref{partially equi-coarse}. By the conclusion of that  lemma, it follows that for every $\beta\in C_{2n}$ the function $f_\beta:(\mathrm{S}_\beta(\mathbb{Z}),d_\infty)\to \tilde M$ is a $(\rho,\omega)$-coarse embedding on every subset of $(\mathrm{S}_\beta(\mathbb{Z}),d_\infty)$ with diameter at most $2n$, and also because $\beta\in C_0$ we have $f_\beta(0) = m_0$.

Fix now $\alpha <\omega_1$. Because  $C_{2n}$ is uncountable we may pick $\beta \in C_{2n}$ so that $\beta>\alpha$. Then it is well known that there exists an infinite subset $L = \{\ell_i:\in\mathbb{N}\}$ of $\mathbb{N}$ so that for all $G = \{a_1,\ldots,a_d\}\in\mathrm{S}_\alpha$ the set $\{\ell_{a_1},\ldots,\ell_{a_d}\}$ is in $\mathrm{S}_\beta$. It follows that the map $s_L:(\mathrm{S}_\alpha(\mathbb{Z}),d_\infty)\to (\mathrm{S}_\beta(\mathbb{Z}),d_\infty)$ given by $\sum_{i\in G}m_ie_i\mapsto\sum_{i\in G}m_ie_{\ell_i}$ is an isometric embedding and maps $0$ to $0$. Therefore, $F_{\alpha,n} = B_{\alpha,n}\to M$, the restriction of $f_\beta\circ s_L$ to $B_{\alpha,n}$, has the desired properties.

Next, we will use Proposition \ref{well founded if countable index} to show that for all $N\in\mathbb{N}$ there exists a function $F_N:B_N\to M$ that is a $(\rho,\omega)$-coarse embedding with the additional property that $F_N(0) = m_0$.  More precisely, we will define this $F_N$ on the subset $B({N,\mathbb{Z}})$ of $B_N $ consisting of all integer valued sequences in the set $B_{N}$. Because this is a 1-net of $B_{N}$ and $N$ is arbitrary we may then deduce the desired conclusion. We denote $\mathbb{I}_N = [-N,N]\cap\mathbb{Z}$ and consider the closed $\mathbb{I}_N$-vine define by
\[
\begin{split}
\mathcal{V} := \Big\{(x_f)_{f\in [\mathbb{I}_N,G]}:&\; G\in[\mathbb{N}]^{<\omega}, x_f\in M \text{ for }f\in[\mathbb{I}_N,G],\; x_0 = m_0,\text{ and for }f,g\in[\mathbb{I}_N,G]\\
&\text{ we have }\rho(\|f-g\|_\infty) \leq d(x_f,x_g)\leq \omega(\|f-g\|_\infty)\Big\}
\end{split}
\]
Because for each $\alpha<\omega_1$ the space $(\mathrm{S}_\alpha(\mathbb{Z}),d_\infty)$ $(\rho,\omega)$-embeds into $(M,d)$ via $F_{\alpha,N}$, which maps $0$ to $m_0$, it follows that $o(\mathcal{V})\ge \omega_1$. Because we are only considering coarse embeddings we may assume that $(M,d)$ is complete. By Proposition \ref{well founded if countable index} the $\mathbb{I}_N$-vine $\mathcal{V}$ must be ill founded, \ie there exists a strictly increasing sequence of integers $(k_m)_m$ and for every finitely supported $f:\{k_i:i\in\mathbb{N}\}\to \mathbb{I}_N$ there exists $x_f\in M$ so that $\chi_m = (x_f)_{f\in[\mathbb{I}_N,\{k_1,\ldots,k_m\}]}$ is in $\mathcal{V}$ for all $m\in\mathbb{N}$. By the definition of $\mathcal{V}$ it follows that the map from $B({N,\mathbb{Z}})$ to $(M,d)$ given by $\sum_{i=1}^\infty m_ie_i\mapsto x_f$, where $f:\{k_i:i\in\mathbb{N}\}\to I_n$ is the function with $f(k_i) = m_i$, is a $(\rho,\omega)$-embedding.
\end{proof}

The last result of this section is a variation of the barycentric gluing technique, which has an interest on its own.  With this gluing technique we can show that if we can equi-coarsely embed the bounded subsets of $\co$ (or equivalently every separable bounded metric spaces) into a metric space $M$ then $M^4$ is coarsely universal. In particular, an immediate consequence of Theorem \ref{T:4.2}  and Theorem \ref{T:4.4}, below,  is

\begin{cor}
Let $(M,d)$ be a separable metric space, If for every $\alpha<\omega_1$ the metric space $(\mathrm{S}_\alpha(\mathbb{Z}),d_\infty)$ embeds coarsely into $(M,d)$, then $\co$ coarsely embeds into $M^4$.
\end{cor}

The original barycentric gluing technique (see \cite{Baudier2007}), creates a coherent embedding of a metric space into a Banach space, by pasting embeddings of balls of growing radii together. Here, the process is reversed in the sense that we will paste balls of Banach spaces into metric spaces, but our proof has the caveat that it requires the  gluing into $M^4$, rather than in $M$. Here is our general result.

\begin{thm}\label{T:4.4}
Let $(X,\|\cdot\|)$ be a Banach space and $(M,d)$ be a metric space. Assume that there exist increasing functions $\rho,\omega\colon [0,\infty)\to [0,\infty)$ that are tending to $\infty$ at $\infty$, $m_0\in M$, and for all $n\in\N$, maps $h_n: nB_X:\to M$, so that $h_n(0)=m_0$, and for all $x,y\in nB_X$
\begin{equation}
\rho(\|x-y\|) \le d\big(h_n(x), h_n(y)\big)\le  \omega(\|x-y\|).
\end{equation}
We equip $M^4$ with the $\ell_\infty$-metric associated with $d$, that we still denote $d$. Then there is a $(\tilde{\rho},\tilde{\omega})$-embedding of $X$ into $M^4$ where
$\tilde{\rho}(t)=\frac12\rho\Big(\frac{t}2\Big)$ and $\tilde{\omega}(t) =8 \omega(3t)$, for all $0<t<\infty$.
\end{thm}

\begin{proof}
Choose inductively $r_0=0<r_1<r_2<\ldots$ in $\N$, so that
  \begin{equation}\label{E:1}
  \rho(r_{n+1})> 2\omega(r_{n}) \text{ and }r_{n+1}\ge 2r_{n}
  \end{equation}

  For $n\in \N$ we define the following map $\alpha_n:[0,\infty)\to [0,1]$
  (set $r_{-4}=r_{-3}=r_{-2}=r_{-1}=r_0=0$.)

  $$\alpha_n(t)=\begin{cases}   0 &\text{if $t<r_{n-4}$ or $t>r_{n}$,} \\
                                            \frac{ t-r_{n-4}}{r_{n-3}-r_{n-4}}                                               &\text{if $r_{n-4}\le t< r_{n-3}$,}\\
                                               1                                                                                             &\text{if $r_{n-3}\le t\le r_{n-1}$,}\\
                                            \frac{r_{n}-t}{r_{n}-r_{n-1}}                                                       &\text{if $r_{n-1}< t\le r_{n}$.}\\
                                                                                 \end{cases}$$
  The support of $\alpha_n$ is $(r_{n-4}, r_n)$, and $\{t: \alpha_n(t)=1\}=[r_{n-3},r_{n-1}]$.

  For $i\in\{0,1,2,3\}$  we define $F^{(i)}: X\to M$ as follows: For $x\in X$ we choose $l\in \Z^+$,  so that
$r_{4(l-1)+i}\le  \|x\|<r_{4l+i}$ and put
$$F^{(i)}(x)=h_{r_{4l+i}}\big(\alpha _{r_{4l+i}}(\|x\|) x\big).$$
Then we define the map
$$F:X\to  M^4 , \quad x\mapsto \big(F^{(0)}(x),F^{(1)}(x),F^{(2)}(x),F^{(3)}(x)\big).$$
We will show that the map $F$ from $X$ into $M^4$, satisfies:
\begin{equation}
\frac{1}{2}\rho\Big(\frac{\|x-y\|}{2}\Big) \le d\big(F(x), F(y)\big)\le  3\omega(3\|x-y\|).
\end{equation}

Firstly, we estimate  the compression function. Let $x,y\in X$, and assume without loss of generality that $\|x\|\le \|y\|$. Choose $l\in\N_0$ and $i\in\{0,1,2,3\}$, so that  $r_{4l+i-2}\le \|y\|\le r_{4l+i-1}$. It is sufficient to show that $d(F^{(i)}(y),F^{(i)}(x))\ge \tilde\rho(\|x-y\|)$. We first note that $\alpha_{4l+i}(y)=1$ and thus $F^{(i)}(y)=h_{r_{4l+i}}(y)$. We consider two cases.

   \noindent Case 1.   $r_{4l+i-3}\le \|x\|$, thus $\alpha _{4l+i}(x)=1$ and  $F^{(i)}(x)=h_{r_{4l+i}}(x)$. It follows that
   $$ d\big(F^{(i)}(x),F^{(i)}(y)\big)=  d\big((h_{r_{4l+i}}(x),h_{r_{4l+i}}(y)\big)\ge \rho(\|x-y\|).$$

   \noindent Case 2. $\|x\|< r_{4l+i-3}$. Thus, for some $m\le l$
   \begin{align*}
    d\big(F^{(i)}(x),F^{(i)}(y)\big)&\ge  d\big(F^{(i)}(y), m_0\big) -d\big(F^{(i)}(x),m_0\big)\\
 &= d\big(h_{r_{4l+i}}(y), h_{r_{4l+i}}(0)\big) -d\big(h_{r_{4m+i}}(\alpha_{4m+i}(\|x\|) x), h_{r_{4m+i}}(0) \big)\\
 &\ge \rho(\|y\|)-\omega(\|x\|)\\
 &\ge \frac12 \rho(\|y\|)+ \frac12\rho(r_{4l+i-2})-\omega(r_{4l+i-3})\\
 &\ge \frac12 \rho(\|y\|)
  \ge \frac12 \rho\Big(\frac{\|x-y\|}{2}\Big).
   \end{align*}

Secondly, we estimate the expansion function. We fix $i\in \{0,1,2,3\}$ and we consider three cases.

 \noindent{Case 1.} For some $n\in\N$ we have $r_{n-1}\le \|x\|\le\|y\| \le r_n$.

 \noindent If $n=4l+i-1$ or $n=4l+i-2$, then $\alpha_{4l+i}(\|y\|)=\alpha_{4l+i}(\|x\|)=1$, and therefore
 $$d\big(F^{(i)}(x),F^{(i)}(y)\big)= d\big( h_{r_{4l+i}}(x),  h_{r_{4l+i}}(y)\big)\le \omega(\|x-y\|).$$
 If $n=4l+i-3$, or $n=4l+i$, then
 $$\big|\alpha_{4l+i}(\|x\|) -\alpha_{4l+i}(\|y\|)\big|=\Bigg|\frac{\|x\|-\|y\|}{r_n-r_{n-1}} \Bigg|\le\frac2{r_n} \|x-y\| ,$$
  and therefore
  $$
  \big\|\alpha_{4l+i}(\|x\|)x-\alpha_{4l+i}(\|y\|) y\big\| \le \alpha_{4l+i}(\|x\|)\|x-y\|+ \|y\| \big| \alpha_{4l+i}(\|x\|)-\alpha_{4l+i}(\|y\|)\big|\le
  3\|x-y\|,$$
 which implies that
 $$d\big(F^{(i)}(x),F^{(i)}(y)\big)  =  d\big( h_{r_{4l+i}}(\alpha_{4l+i}(\|x\|)x),  h_{r_{4l+i}}(\alpha_{4l+i}(\|y\|)y)\big) \le\omega(3\|x-y\|).$$

From now we assume that there are $m,n\in\N$, $m<n$, so that $r_m\le \|x\|\le r_{m+1}\le r_n\le \|y\|\le r_{n+1}$.
For  $j=1,2,\ldots,n-m$,  let $z_j$ be the  element on the the segment $[x,y]$  (\ie points of the form $x+t(y-x)$ with $0\le t\le 1$) so that $\| z_j\|=r_{m+j}$, and put $z_0=x$ and $z_{n-m+1}= y$.

 \noindent Case 2.  $n-m\le 3$.\\
$$
d\big(F^{(i)}(x),F^{(i)}(y)\big)\le \sum_{i=1}^{n-m+1} d\big(F^{(i)}(z_{i-1}),F^{(i)}(z_{i})\big)
\le \sum_{i=1}^{n-m+1} \omega(3\|z_{i-1}-z_{i}\|) \le 4\omega(3\|x-y\|).
$$

\noindent Case 3. $n-m\ge 4$. It follows then from Case 2 that
\begin{align*}
d\big(F^{(i)}(x),F^{(i)}(y)\big)&\le d\big(F^{(i)}(x),m_0\big)+d\big(F^{(i)}(y),m_0\big)\\
&=d\big(F^{(i)}(x),F^{(i)}(z_{j_1})\big)+d\big(F^{(i)}(y),F^{(i)}(z_{j_2})\big)\\
&\le 4\omega(3\|x-z_1\|) +4\omega(3\|y-z_2\|)\le 8\omega(3\|x-y\|).
\end{align*}

\end{proof}

\section{Universality properties of interlacing graphs}
In \cite{Kalton2007}, Kalton showed that a Banach space $X$ that is coarsely universal for the class of all separable metric spaces, or equivalently that coarsely contains $\co$, cannot have separable iterated duals, \ie $X^{(r)}$ is nonseparable from some $r\ge 2$. Kalton's argument is based on the metric properties of the interlacing graphs. As we will see in the next section, these graphs introduced by Kalton have some remarkable universality properties.

\subsection{Almost isometric universality of the interlacing graphs}
We define a slightly larger class of interlacing graphs than the ones introduced by Kalton. The set of vertices is $[\N]^{<\omega}$, the set of finite subsets of $\N$, and we declare that two vertices $A = \{a_1,\ldots,a_n\}$ and $B = \{b_1,\ldots,b_m\}$ in $[\N]^{<\omega}$
are adjacent if and only if $a\not=b$ and  one of the following interlacing relations holds
\begin{itemize}
 \item[(i)] $n = m+1$ and $a_i\leqslant b_i \leqslant a_{i+1}$ for $1\leqslant i \leqslant m$,
 \item[(ii)] $m = n+1$ and $b_i\leqslant a_i \leqslant b_{i+1}$ for $1\leqslant i \leqslant n$,
 \item[(iii)] $n = m$, $a_i\leqslant b_i \leqslant a_{i+1}$ for $1\leqslant i < n$, and $a_n\leqslant b_n$, or
 \item[(iv)] $n = m$, $b_i\leqslant a_i \leqslant b_{i+1}$ for $1\leqslant i < n$, and $b_n\leqslant a_n$.
\end{itemize}
We also connect the empty set with all singletons. We refer to this graph as the universal interlacing graph, and we denote $([\N]^{<\omega},d_{\mathrm{I}})$ the universal interlacing graph equipped with its canonical graph metric. Kalton's interlacing graph are defined in the same way besides only vertices with the same length were considered. More precisely, Kalton's interlacing graph of diameter $k$ is the space $([\N]^{k},d_{\mathrm{I}}^{(k)})$, where the graph metric only refers to the interlacing relations $(iii)$ or $(iv)$ in this case. For $A, B\in[\N]^k$, although it is obvious that $d_{\mathrm{I}}^{(k)}(A,B) = 1$ if and only if $d_{\mathrm{I}}(A,B) = 1$, it is not immediately clear that on $[\N]^k$  the metrics $d_{\mathrm{I}}^{(k)}$ and $d_{\mathrm{I}}$ coincide. As we shall see later, this is indeed the case.

The universality properties of the interlacing graphs stem from the fact that the interlacing metric admits an interpretation in terms of the summing norm on $\co$. For $A, B$ in $[\N]^{<\omega}$ define the summing distance $$d_{\mathrm{sum}}(A,B) = \left\|\sum_{i\in A}s_i - \sum_{i\in B}s_i\right\|_{\mathrm{sum}}$$
where $(s_i)_i$ denotes the summing basis of $c_0$, endowed with the usual bimonotone version of the summing norm, i.e.
\begin{align}
\label{summing norm}
\left\|\sum_ia_is_i\right\|_{\mathrm{sum}}= \sup\left\{\left|\sum_{i=k}^ma_i\right|: k,m\in \N,\,  k\le m\right\}.
\end{align}
 In \eqref{summing norm} one only needs to consider intervals at whose boundaries are sign-changes of the $a_i$'s. More precisely
for a sequence $(a_i)$ in $c_{00}$ let $0=m_0<m_1<\ldots m_s$ be chosen in $\N$ so that for all $i\le s$ the signs of $a_j$ on $j\in [m_{i-1}+1, m_i]$ are the same
(\ie all non negative or all non positive) then
\begin{align}
\label{summing norm2}
\left\|\sum_ia_is_i\right\|_{\mathrm{sum}}= \sup\left\{\left|\sum_{i=k}^l \sum_{m_{i-1}+1}^{m_i} a_j\right|   1\le k\le l\le s\}\right\}.
\end{align}
Thus for $A,B\subset [\N]^{<\omega}$ we write $A\triangle B$ in increasing order as $A\triangle B=\{x_1,x_2,\ldots,x_n\}$ and note that
\begin{align}
\label{summing metric for sets}
d_\s(A,B) &= \max\left\{\left|\#(A\cap E) - \#(B\cap E)\right|: E \text{ is an interval of }\N\right\}\\
&= \max\left\{\left|\#(A\cap [x_i, x_j]) - \#(B\cap[x_i,x_j])\right|:  1\le i\le j\le n \right\}.\notag
\end{align}
The above forms of the metric $d_\s$ will be used more often. We first show that the interlacing metric and the summing distance coincides. For fixed $k$, the coincidence of $d^{(k)}_I(A,B)$ with $\max\left\{\left|\#(A\cap E) - \#(B\cap E)\right|: E \text{ is an interval of }\N\right\}$ was already shown in \cite{LancienPetitjeanProchazka2019}, where it was afterwards used in connection with the canonical norm of $c_0$ instead of $\|\cdot\|_{\s}$.

For $n\in\N$, $A=\{a_1,a_2,\ldots , a_n\}\in\N^{<\omega_1}$, with   $a_1<a_2<\ldots a_n$, we call $A'=\{ a'_1, a'_2, \ldots a'_n\}$, with $a'_1<a'_2<\ldots<a'_n$
{\em a shift to the left of $A$ } if  $A' \not= A$ and $a_1\le a'_1\le a_2\le a'_2\le \ldots\le a_n\le a'_n$.
We note that in this case

\begin{equation}\label{E:5.2.1} d_I(A,A')=d_I(A, A'\setminus \{a'_n\}) =d_{\s}(A,A')=d_{\s}(A, A'\setminus \{a'_n\}).\end{equation}

For another set $B\in[\N]^{<\omega}$
we say that a left shift $A'$ of $A$ is a {\em shift towards $B$} if $A'\setminus A\subset B\setminus A$.

\begin{thm}\label{T:5.2.1} For $A,B\in[\N]^{<\omega}$ we have that $d_\s(A,B) = d_I(A,B)$.

Moreover if $k=\#A=\#B$ then there is a path of length $d_I(A,B)$ from $A$ to $B$ in the interlacing graph,  which stays in $[\N]^k$. Thus the restriction of $d_I$ to  $[\N]^k$ is $d^{(k)}_I$.

\end{thm}
\begin{proof} We prove our statement by induction for all $m\in\N\cup\{0\}$, and all $A,B\in[\N]^{<\omega}$ with $m=d_\s(A,B)$.

If $m=0$ and $d_\s(A,B)=0$ and thus $A=B$, our claim is trivial. If $m=1$ and $d_\s(A,B) = 1$, we will show that $d_I(A,B) = 1$. Write $A = \{a_1,\ldots,a_n\}$, $B = \{b_1,\ldots,b_m\}$ and assume, without loss of generality, that $\min(A\triangle B) = a_{i_0}\in A$.  Note that $|n - m| = |\#A -\#B| \leq d_\s(A,B) = 1$ and by the assumption  $\min(A\triangle B) = a_{i_0}\in A$ we have $1 \leq \#A\cap[a_{i_0},\max\{A\cup B\}] - \#B\cap[a_{i_0},\max\{A\cup B\}] = (n-i_0+1) - (m - i_0)$, i.e., $m\leq n\leq m+1$. Next, observe that for $1\leq i\leq \min\{m,n\}$ we have $a_i\leq b_i$. Otherwise, set $j_0 = \min\{1\leq i\leq \min\{m,n\}: a_i>b_i\}$ and note that $a_{i_0} < b_{i_0}$ and thus $i_0<j_0$. If we set $E = [b_{i_0},b_{j_0}]$ then $d_\s(A,B) = 1 \geq \#B\cap E - \#A\cap E = (j_0 - i_0+1) - (j_0 - i_0 - 1) = 2$. We also observe that for $1\leq i\leq \min\{m,n-1\}$ we have $b_i\leq a_{i+1}$. If this is not the case, set $s_0 = \min\{1\leq i\leq \min\{m,n-1\}: b_i>a_{i+1}\}$ and observe that if $E = [a_1,a_{s_0+1}]$ then $\#A\cap E = s_0 + 1$ whereas $\#B\cap E = s_0 -1$, which is absurd. Finally we distinguish the cases $n=m$ and $n=m+1$. If $n=m$ then we have demonstrated that (iii) of the definition of adjacency holds. If $n=m+1$ then we have demonstrated that (i) holds.

Assume now that for some $m\geq 2\in \N$, and all $A,B\in[\N]^{<\omega}$  with
  $d_\s(A,B)<m$ it follows that
 $d_\s(A,B) = d_I(A,B)$, and that, $d_I(A,B)=d_I^{(k)}(A,B)$ if $k=\#A=\#B$.

 Let $A,B\in[\N]^{<\omega}$  with $d_\s(A,B)=m$.
  If $A\subset B$, or $B\subset A$, and we   assume without loss of generality
  that $B\subsetneq A$, we put  $A'=A\setminus \{ a\}$, where $a\in A\setminus B$.
  Then $d_{\s}(A, A')=d_{I}(A,A')=1 $ and $d_\s(A',B)=m-1$, and we deduce our claim from the induction hypothesis

   Assuming that $A\not\subset B$ and  $B\not\subset A$
   we write
 $A\triangle B$ in increasing order as
  $A\triangle B=\{x_1,x_2,\ldots,x_l\}$.
  It   follows that 

  $$m=d_\s(A,B)= \max_{i\le j}\big| \#(A\cap[x_i,x_j])-\#(B\cap[x_i, x_j])\big|. $$
Without loss of generality we can assume that
  $x_1\in A\setminus B$. There is a $t\in\N$
  and numbers $1\le i_1<i_2<\ldots <i_t<l$ so that
  $$\{i_s: s=1,\ldots t\} = \big\{i\in\{1,2,\ldots, l-1\}:  x_i\in A\text{ and } x_{i+1} \in B\big\}.$$

  We will  now define $A'\in[\N]^\omega$ for which
  $d_I(A,A')=d_{\s} (A,A')=1 $ and
  $d_{\s}(A', B)\le m-1$, and consider the following two cases:


       \noindent Case 1.
   For all $1\le j\le l$  we have $\#(A\cap[x_j,x_l])-\#(B\cap[x_j,x_l])<m$. Then we put
    $$A' =(A\setminus \{x_{i_s}: s\le t\})\cup\{ x_{i_s+1} : s\le t\},$$
    which is a left shift of $A$ towards $B$.

     \noindent Case 2.
      There is  a $j\le l $ so that $\#(A\cap[x_j,x_l])-\#(B\cap[x_j,x_l])=m$.  It follows that $x_l\in A$ and we put
   $$A' =\Big((A\setminus \{x_{i_s}: s\le t\})\cup\{ x_{i_s+1} : s\le t\}\Big)\setminus\{x_l\}.$$

We observe that if $\#A=\#B$ the second case cannot happen. Indeed,
assume that there is a $j\le l$ so that  $\#(A\cap[x_j,x_l])-\#(B\cap[x_j,x_l])=m$, then $j>1$ and
 it follows that
\begin{align*}
 \#(B\cap[x_2,x_{j-1}] )- \#(A\cap[x_2,x_{j-1}])&=
 \#(B\setminus A)- \#(B\cap(\{x_1\}\cup[x_j,x_l])-( \#(A\setminus B)- \#(A\cap(\{x_1\}\cup[x_j,x_l]))\\
&= \#(A\cap(\{x_1\}\cup[x_j,x_l]))- \#(B\cap(\{x_1\}\cup[x_j,x_l]))=m+1
\end{align*}
which is a contradiction.

Thus, it follows  $\#A'=\#A$ if $\#A=\#B$.

From \eqref{E:5.2.1}  it follows that $d_I(A,A')=d_{\s}(A,A')=1$.
  We need to show that
  $d_{\s}(A',B)\le m-1$, and thus, by the triangle inequality  $d_{\s}(A',B)= m-1$.

  First let $i\in\{1,2\ldots,l\}$ and define
  for $i\leq j\leq l$
  $$f(j)= \#(A'\cap[x_i,x_j])-\#(B\cap[x_i, x_j]). $$
  Observe that $f(i) \leq 1 \leq m-1$. We claim that for all $i<j\le l$
  \begin{equation} f(j)\le \min \big( \#(A\cap[x_i,x_j])-\#(B\cap[x_i, x_j]), m-1\big).
  \end{equation}
    Since $A'\setminus B\subset A\setminus B$ we have   $f(j)\le  \#(A\cap[x_i,x_j])-\#(B\cap[x_i, x_j])$, for $i\le j$.

  Assume that our  claim is not true, and let $k$  be the minimum of all $j>i$ so that
  $$f(j)=1+\min\big( \#(A\cap[x_i,x_j])-\#(B\cap[x_i, x_j]), m-1\big).$$
  Since $f(k)\le f(k-1) +1$ it follows
  that $\#(A'\cap[x_i,x_{k-1}])-\#(B\cap[x_i, x_{k-1}])=m-1$ and
  since $A'\triangle B\subset A\triangle B$ it follows that $x_{k}\in A$ and thus
   $\#(A\cap[x_i,x_{k-1}])-\#(B\cap[x_i, x_{k-1}])=m-1$
   (otherwise $\#(A\cap[x_i,x_{k-1}])-\#(B\cap[x_i, x_{k-1}])=m$ and thus
 $\#(A\cap[x_i,x_{k}])-\#(B\cap[x_i, x_{k}])=m+1$ ).
 It follows  therefore that  $\#(A\cap[x_i,x_{k}])-\#(B\cap[x_i, x_{k}])=m$.

 Either $k<l$ then $x_{k+1}\in B$, and since $x_k\in A$ it follows from the definition of $A'$
 that $x_k\not\in A'$ and thus
 $f(k)=f(k-1)=m-1$, which contradicts our assumption.
 Or $k=l$ and thus $\#(A\cap[x_i, x_l])- \#(B\cap[x_i, x_l])=m$, which
 implies  that $x_k=x_l\not\in A'$, since the second case in the definition of $A'$ occurs,
 it would  again follow that $f(k)=m-1$, which is also a  contradiction.

 Next we let $j=1,\ldots,l$, and put for $i=1,2,\ldots j$
 $$g(i)=  \#(B\cap[x_i,x_j])-\#(A'\cap[x_i, x_j]),$$
 and claim
 that $g(i)\le  \min(\#(B\cap[x_i,x_j])-\#(A\cap[x_i, x_j]), m-1)$ for all $i\in\{1,2,\ldots, l\}$.

 Again since $B\triangle A'\subset B\triangle A$ it follows that $g(i)\le  \#(B\cap[x_i,x_j])-\#(A\cap[x_i, x_j])$ for all $i\in\{1,2,\ldots, j\}$.

Assume our claim is not true and  let $k$ be the maximal $k<j $ so that
$g(k)=m$. So it follows that $\#(B\cap[x_k,x_j])-\#(A\cap[x_k, x_j])=m$,
and thus $x_k\in B$, and $x_{k-1}\in A$ (note that $k\not=1$  since $x_1\in A$)
But this means that $x_k\in A'$ and thus
 $\#(B\cap[x_k,x_j])-\#(A'\cap[x_k, x_j])=m-1$, which is again a contradiction.

We therefore showed that for all $1\le i<j\le l$, $|\#(A'\cap[x_i,x_j])- \#(B\cap[x_i,x_j])|\le m-1$,  which finishes our proof.

\end{proof}

The following Corollary could of course be also proven directly very easily.

\begin{cor}
\label{increasing length isometrically}
For all $k,m\in\N$ with $k<m$, $([\N]^{k},d_{\mathrm{I}}^{(k)})$ embeds isometrically into $([\N]^{m},d_{\mathrm{I}}^{(m)})$.
\end{cor}

The following quantitative embedding result immediately implies Theorem \ref{T:D}.

\begin{thm}
\label{univeral for finite}
Let $(X,d)$ be a $n$-point metric space, and $\alpha:=\alpha(X)=\frac{\mathrm{diam}(X)}{\mathrm{sep}(X)}$ be its aspect ratio, where $\mathrm{diam}(X)=\sup\{d_X
(x,y)\colon x,y\in X\}$ and $\mathrm{sep}(X)=\inf\{d_X(x,y)\colon x,y\in X\}$. Then for every $0<\e<1$ and every integer $k\geqslant \big(n + \frac32\big)\big(\frac{\alpha}{\e} + \diam(X) + 1\big)$, $(X,d)$ embeds with distortion $(1-\e)^{-1}$ into $([\N]^{k},d_{\mathrm{I}}^{(k)})$.

In particular, for all $\e>0$ $(X,d)$ embeds with distortion at most $1+\e$ into $([\N]^{<\omega},d_{\mathrm{I}})$.
\end{thm}

\begin{proof}
The proof of the general situation can be reduced to the special case where $(X,d)$ is a finite metric space with even distances. Assuming that we have proven
\begin{claim}\label{Cl:1}
Assume that for all $x,y\in X$, $d(x,y)$ is an even integer and that $k$ is an integer number such that $k \geqslant \frac12(n+\frac32)(\mathrm{diam(X) + 2})$. Then, the space $(X,d)$ embeds isometrically into $([\N]^{k},d_{\mathrm{I}}^{(k)})$.
\end{claim}
Then we can finish the proof of the general case.
Indeed, let $\vp>0$ and  choose an integer $q$ such that $(\textrm{sep}(X)\e)^{-1}\le q \le (\textrm{sep}(X)\e)^{-1}+1$ and for each $x, y\in X$ define $k_{x,y} = \max\{k\in\N\cup\{ 0\}: k/q\leqslant d(x,y)\}$. Define a metric $\tilde d$ on $X$ with
$$\tilde d(x,y) = \min\left\{\sum_{i=1}^\ell\frac{k_{x_i,x_{i-1}}}{q}: x_i\in X \text{ for } 0\leqslant i\leqslant \ell\text{ and } x=x_0,y = y_\ell\right\}.$$
One can check that $\tilde d$ is indeed a metric and for all $x,y\in X$ we have $$(1-\e)d(x,y)\leqslant \tilde d(x,y)\leqslant d(x,y),$$
hence, it suffices to embed  the space $(X,\tilde d)$ with distortion $1$ into $([\N]^{k},d_{\mathrm{I}}^{(k)})$ for appropriate $k$.
Note that if we denote $\tilde X=(X,\tilde d)$, $\diam(\tilde X)\leqslant \diam(X)$. By Claim \ref{Cl:1}, the space $(X,2q\tilde d)$
embeds isometrically into $([\N]^{k},d_{\mathrm{I}}^{(k)})$, for $k\geqslant \frac12(n + \frac32)(2q\diam(X) + 2)$. Recall that $q \leqslant (\textrm{sep}(X)\e)^{-1}+1$, which implies that
$\frac12(n + \frac32)(2q\diam(X)+ 2) \leqslant (n + \frac32)(\frac{\alpha}{\e} + \diam(X) + 1)$.\end{proof}

\begin{proof}[Proof of Claim \ref{Cl:1}]
We will find an embedding $\Phi$ from $X$ into the linear span of $(s_i)_i$, endowed with the norm \eqref{summing norm}, so that for each $x\in X$
the vector $\Phi(x)$ is of the for form $\sum_{i\in A(x)}s_i$, with $\#A(x)\le  (1/2)(n+3/2)(\mathrm{diam(X) + 2})$.

We enumerate $X = \{x_2,\ldots,x_{n+1}\}$. We add one more point $x_1$ to obtain the set $\tilde X = \{x_1,x_2,\ldots,x_{n+1}\}$. We extend $d$ on $\tilde X$ by setting for  $x\in X$, $d(x_1,x) = d(x,x_1) = D$,
where $D$ is the minimal even integer with $2D\geqslant \mathrm{diam}(X)$. As the diameter of $X$ is an even integer, we deduce $D \leqslant \mathrm{diam}(X)/2+1$.
It is straightforward to verify that the triangle inequality is still satisfied on $\tilde X$.
For notational reasons, we add a ``ghost'' point $x_{n+1}$ with the property $d(x,x_{n+1}) = 0$ for all $x\in X$.
We first define a map $\Phi_0:X\to\langle\{s_i: i\in\N\}\rangle$, the linear span of the $s_i$'s, by
\begin{equation*}
\Phi_0(x) = \frac{1}{2}\sum_{i=1}^{n+1}\left(d(x,x_i)-d(x,x_{i+1})\right)s_i.
\end{equation*}
If we denote by $(s_i^*)_i$ the sequence of coordinate functionals associated to $(s_i)_i$ we observe that for all $i\in\N$ and $x\in X$, the number $s_i^*(\Phi_0(x))$ is an integer.
We start by showing that $\Phi_0(x)$ is an isometric embedding. Let $1\leqslant k\leqslant m\leqslant n+1$ be such that
\begin{align*}
\left\|\Phi_0(x) - \Phi_0(y)\right\| & = \frac{1}{2}\left|\sum_{i=k}^m\left(d(x,x_i) - d(x,x_{i+1} - d(y,x_i) + d(y,x_{i+1})\right)\right|\\
&=\frac{1}{2}\left|d(x,x_{m+1}) - d(y,x_{m+1}) -d(x,x_k) + d(y,x_k)\right|\\
&\leqslant \frac{1}{2}\left|d(x,x_{m+1}) - d(y,x_{m+1})\right| + \frac{1}{2}\left|d(x,x_k) - d(y,x_k)\right|
\leqslant d(x,y).
\end{align*}
For the inverse inequality, let $x = x_j$, $y = x_{j'}$ and let us assume without loss of generality $j < j'$. Define $k = j$ and $m = j' - 1$. Then,
\begin{align*}
\left\|\Phi_0(x) - \Phi_0(y)\right\| & \geqslant \frac{1}{2}\left|\sum_{i=k}^m\left(d(x,x_i) - d(x,x_{i+1}) - d(y,x_i) + d(y,x_{i+1})\right)\right|\\
&= \frac{1}{2}\left|d(x,x_{m+1}) - d(y,x_{m+1}) -d(x,x_k) + d(y,x_k)\right|\\
&= \frac{1}{2}\left|d(x,y) - d(y,y) - d(x,x) + d(y,x)\right| = d(x,y).
\end{align*}

Define $\Phi_1(x) = \Phi_0(x) + D\sum_{i=1}^{n+1}s_i$.
Then $\Phi_1$ is an isometric embedding of $X$ into $\langle\{s_i: i\in\N\}\rangle$ so that for all $i\in\N$ and $x\in X$ the number $s_i^*(\Phi_1(x))$ is a non-negative integer.
For $k=1,\ldots,n+1$ define $$N_k = \max\{s_k^*(\Phi_1(x)): x\in X\}\text{ and }M_k = \sum_{j=1}^kN_j.$$ Also define $M_0 = 0$. We are ready to define the desired embedding.
For $x\in X$ set
$$\Phi(x) = \sum_{k=1}^{n+1}\sum_{\left\{\substack{M_{k-1} < i\leqslant M_{k-1}\\+s_k^*(\Phi_1(x))}\right\}}s_i.$$
We deduce that  $\Phi(x)$ is of the form $\sum_{i\in A(x)}s_i$ with
\begin{align*}
\#A(x) &=\sum_{k=1}^{n+1}s_k^*\left(\Phi_1(x)\right) = \frac{1}{2}\sum_{k=1}^{n+1}\big(d(x,x_i)-d(x,x_{i+1})\big) + D(n+1)\\
&=\frac{1}{2}\left(d(x,x_1) - d(x,x_{n+1})\right) + D(n+1) = \frac{1}{2}D + D(n+1)\\
&=\big(n+\frac32\big)D \leqslant \frac12\big(n+\frac32\big)\big(\mathrm{diam(X) + 2}\big)
\end{align*}
Applying   \eqref{summing norm2} to  $m_i=M_i$,  $i=1,2 \ldots, n$ we deduce for $x,y\in X$ that
\begin{align}\left\|\Phi(x) - \Phi(y)\right\|&=\max\{\Big| \sum_{i=p}^q\sum_{j=M_{i-1}+1}^{M_i}   s^*_j(\Phi(x) - \Phi(y))\Big|:  1\le p\le q\le n\}\\
&=\max\Big\{ \sum_{i=p}^q  s^*_i(\Phi_0(x) - \Phi_0(y))\Big\}=d(x,y).
\end{align}
Finally our conclusion follows therefore from Theorem \ref{T:5.2.1}    and Corollary \ref{increasing length isometrically}.
\end{proof}

\subsection{Metric universality and metric elasticity}
It is a well known and long standing open problem whether $\co$ isomorphically embeds into a Banach space whenever it bi-Lipschitzly embeds into it. Due to Aharoni's theorem this fundamental rigidity problem in nonlinear Banach space geometry can be reformulated as the following universality question.

\begin{prob}\label{P:5.9}
Let $X$ be a Banach space.
If $X$ is Lipschitz universal for the class of separable metric spaces, does $X$ contain an isomorphic copy of $\co$?
\end{prob}

It is possible to answer positively Problem \ref{P:5.9} for Banach lattices using Kalton's work on the interlacing graphs. This fact seems to have been overlooked and we describe the argument in the ensuing discussion. Recall that a Banach space $Y$ has Kalton's property $\cQ$ if there exists $C\in(0,\infty)$ such that for all $k\in \N$ and every Lipschitz map $f$ from  $([\N]^k, d^{(k)}_{\mathrm{I}})$ to $Y$, there exists $\M\in [\N]^\omega$ such that for all $\mb, \nb \in [\M]^k$ we have
\begin{equation}\label{E:12}
\big\| f(\mb)-f (\nb)\big\|_Y\le C\Lip(f).
\end{equation}
Kalton showed that reflexive Banach spaces \cite[Theorem 4.1]{Kalton2007} and more generally, Banach spaces whose unit ball uniformly embeds into a reflexive Banach space \cite[Corollary 4.3]{Kalton2007} have Property $\cQ$. It follows from \eqref{E:12} (and the fact that coarse embeddings $f$ whose domains are graphs must be $\omega(1)$-Lipschitz) that the sequence of interlacing graphs cannot equi-coarsely embed into a Banach space with property $\cQ$. Therefore if a Banach space $X$ equi-coarsely contains the interlacing graphs, it must fail property $\cQ$. By \cite[Corollary 4.3]{Kalton2007}, the unit ball of $X$ does not uniformly embed into a reflexive Banach space. But Kalton also proved \cite[Theorem 3.8]{Kalton2007}  that for a separable Banach lattice $X$, $B_X$ uniformly embeds into a reflexive Banach space if and only if $X$ does not contain any subspace isomorphic to $\co$. Thus, it follows from the above discussion that:

\begin{thm}[\cite{Kalton2007}]\label{T:5.10} If $X$ is a separable Banach lattice and if $([\N]^k,d_{\mathrm{I}})_k$ equi-coarsely embeds into $X$, then $X$ contains an isomorphic copy of $\co$.
\end{thm}

The following statement is an immediate consequence of Theorem \ref{T:5.10}.

\begin{cor}\label{C:5.11}
If a separable Banach lattice $X$ is coarsely universal for the class of separable metric spaces, then $X$ contains an isomorphic copy of $\co$.
\end{cor}

Thus, Problem \ref{P:5.9} (as well as its coarse analogue) has a positive solution for Banach lattices.
It is worth pointing out that the coarse (resp. uniform) analogue of Problem \ref{P:5.9} does not hold in general since using the theory of H\"{o}lder free spaces, it is proved in \cite{Kalton2004} that $\co$ coarsely (resp. uniformly) embeds into a Schur space. Recall that a Banach space has the Schur property if every weakly null sequence converges to $0$ in the norm topology, and hence a Schur space cannot contain any isomorphic copy of $\co$.
\begin{rem} The Lipschitz version of
Corollary \ref{C:5.11} can be proven for a Banach space with an unconditional basis using classical linear and nonlinear Banach space theory. Indeed, for Banach spaces with an unconditional basis the following dichotomy holds: either the unconditional basis is not boundedly-complete, and by a result of James \cite{James1950} $X$ will contain an isomorphic copy of $\co$, or the unconditional basis is boundedly-complete and hence $X$ will be isomorphic to a dual space and thus $X$ will have the Radon-Nikod\'ym property.
Note that the two possibilities are mutually exclusive. In the first situation the conclusion of Corollary \ref{C:5.11} already holds, and in the second situation we can use classical differentiability theory and obtain a contradiction. A similar dichotomy argument fails for Banach lattices since $L_1$ is a Banach lattice that does not linearly contain $\co$ and yet $L_1$ does not have the Radon-Nikod\'ym property.
\end{rem}

Theorem \ref{T:5.10} has also an application to metric analogues of the linear notion of elasticity. In 1976 Sch\"affer raised the problem whether the isomorphism class of every infinite dimensional Banach space $X$ is unbounded in the sense that $D(X):=\sup\{d_{BM}(Y,Z)\colon Y, Z \textrm{ are isomorphic to } $X$\}=\infty$ where $d_{BM}$ denotes the Banach-Mazur distance \footnote{This widely use terminology can be misleading since $\log(d_{BM})$ (and not $d_{BM}$) is a semi-metric.}. Johnson and Odell introduced the notion of elasticity for their solution of Sch\"affer's problem for \emph{separable} Banach spaces.

\begin{defin}[\cite{JohnsonOdell2005}]\label{D:5.11}
Let $K\in[1,\infty)$. A Banach space $Y$ is \emph{$K$-elastic} provided that if a Banach space $X$ isomorphically embeds into $Y$ then $X$ must be $K$-isomorphically embeddable into $Y$, and $Y$ will be \emph{elastic} if it is $K$-elastic for some $K$.
\end{defin}

The connection with Sch\"affer's problem comes from the observation that if $D(X)<\infty$ then $X$ as well as every isomorphic copies of $X$ is $D(X)$-elastic. Elasticity is intimately connected to universality. First of all, it is immediate that every Banach space is crudely finitely representable into any elastic Banach space, in particular every elastic Banach space has trivial cotype. Second of all, a consequence of Banach-Mazur embedding theorem is that $C[0,1]$ is $1$-elastic. Moreover, a key step in \cite{JohnsonOdell2005} is the following theorem.

\begin{thm}\cite[Theorem 7]{JohnsonOdell2005}\label{T:5.12}
Let $X$ be a separable infinite-dimensional Banach space. If $X$ is elastic then $\co$ embeds isomorphically into $X$.
\end{thm}

The conjecture from \cite{JohnsonOdell2005} that a separable elastic Banach space must contain an isomorphic copy of $C[0,1]$, was recently solved positively by Alspach and Sari \cite{AlspachSari2016}.

We now discuss a metric analogue of Theorem \ref{T:5.12}. According to Johnson and Odell a Banach space $Y$ is said to be \emph{Lipschitz K-elastic} provided that if a Banach space is isomorphic to $Y$ then $X$ must bi-Lipschitzly embed into $Y$ with distortion at most $K$. Johnson and Odell definition of Lipschitz elasticity is motivated by the fact that a Banach space $X$ is $K$-elastic if and only if every isomorphic copy of $X$ is $K$-isomorphic to a subspace of $X$ (the proof uses an Hahn-Banach extension argument that goes back to Pelczy\'nski \cite{Pelczynski1960}). It was observed in \cite{JohnsonOdell2005} that it follows from Aharoni's embedding theorem and James' distortion theorem, that there exists $K\ge 1$ such that every Banach space that contains an isomorphic copy of $\co$ must be Lipschitz $K$-elastic. The constant $K$ is related to the optimal distortion in Aharoni's embedding and can be taken to be $2+\vp$ for every $\vp>0$ due to \cite{KaltonLancien2008}. Motivated by Definition \ref{D:5.11} the following definition is another metric analogue of elasticity.

\begin{defin}
Let $K\in[1,\infty)$. A metric space $Y$ is \emph{metric $K$-elastic} provided that if a metric  space $X$ bi-Lipschitzly embeds into $Y$ then $X$ must be bi-Lipschitzly embeddable into $Y$ with distortion at most $K$, and $Y$ will be \emph{metric elastic} if it is metric $K$-elastic for some $K$.
\end{defin}

It is immediate that a Banach space that is metric $K$-elastic (as a metric space) is Lipschitz $K$-elastic. With this stronger nonlinear notion of elasticity we obtain the following theorem, which contains a strong nonlinear analogue of Theorem \ref{T:5.12} in the context of Banach lattices.

\begin{thm}\label{T:5.14}
Let $X$ be a separable infinite-dimensional Banach lattice. The following assertions are equivalent.
\begin{enumerate}
\item $\co$ isomorphically embeds into $X$.
\item $\co$ bi-Lipschitzly embeds into $X$.
\item $\co$ coarsely embeds into $X$.
\item $X$ is metric elastic.
\item $([\N]^k,d_{\mathrm{I}})_{k\in\N}$ embeds equi-bi-Lipschitzly into $X$.
\item $([\N]^k,d_{\mathrm{I}})_{k\in\N}$ embeds equi-coarsely into $X$.
\end{enumerate}
\end{thm}

\begin{proof}
(1) implies (2) implies (3) is trivial. (3) implies (1) is Corollary \ref{C:5.11}.  (2) implies (4) follows from Aharoni's embedding theorem and the fact that separability is a Lipschitz invariant. For (4) implies (5), observe that an infinite-dimensional Banach space $X$ has a $1$-separated sequence of unit vectors, and thus for all $k\in\N$, the $k$-dimensional interlacing graph $([\N]^k,d_{\mathrm{I}})$ (which is countable, $1$-separated, and has diameter $k$) embeds bi-Lipschitzly into $X$ with distortion at most $k$. Since $X$ is metric $K$-elastic for some $K\ge 1$, it follows that $\sup_{k\in\N}c_X(([\N]^k,d_{\mathrm{I}}))\le K$. For Banach spaces, (5) implies (6) always holds. An appeal to Corollary \ref{T:5.10} gives the remaining implication.
\end{proof}

\subsection{Separating interlacing graphs in Banach spaces with nonseparable biduals}
The following concentration result for interlacing graphs was shown by Kalton  \cite{Kalton2007}.
\begin{thm}\label{T:2}  \cite[Theorem 3.5]{Kalton2007}. Let $k\in \N$ and $Y$ be a Banach space such that $Y^{(2k)}$, the iterated dual of order $2k$ of $Y$, is separable. Assume that $(g_i)_{i\in I}$ is an uncountable family of $1$-Lipschitz maps from  $([\N]^k, d^{(k)}_{\mathrm{I}})$ to $Y$. Then there exist $i\neq j \in I$ and $\M\in [\N]^\omega$ such that for all $\nb \in [\M]^k$ we have
$$\big\| g_i(\nb)-g_j (\nb)\big\|\le 3.$$
\end{thm}

Vaguely speaking, it follows from Theorem \ref{T:2} that if a Banach space $X$ contains uncountably many well-separated $1$-Lipschitz images of the interlacing graphs and if $X$ coarsely embeds into a Banach space $Y$, then $Y$ cannot have all its iterated duals separable. This idea was devised by Kalton in \cite{Kalton2007} to show that if $c_0$ coarsely embeds into a Banach space $Y$, then one of the iterated duals of $Y$ is non separable (in particular, $Y$ cannot be reflexive). It was adapted in \cite{LancienPetitjeanProchazka2019} to show that the same conclusion holds if the James tree space $JT$ or its predual coarsely embeds into $Y$. In these proofs the non separability of the bidual of the embedded space plays an important role. However, since $\ell_1$ coarsely embeds into $\ell_2$, this is not a sufficient condition. We will prove that a certain presence of $\ell_1$ in the embedded space is essentially the only obstruction.

\begin{thm}\label{T:5.8}{\rm[Theorem F]}  Let  $X$ be a separable Banach space with non separable bidual $X^{**}$, $\ell_1\not\subset X$, and such that no spreading model generated by  a normalized  weakly null sequence in $X$  is equivalent to the $\ell_1$-unit vector basis. Assume that $X$ coarsely embeds into a Banach space $Y$. Then there exists $k\in \N$ such that $Y^{(2k)}$ is non separable.
\end{thm}

\begin{proof} We start with the construction of our well separated $1$-Lipschitz maps from $([\N]^k, d^{(k)}_\mathrm{I})$ to $X$. Since $X$ is separable and $X^{**}$ is not, using Riesz Lemma and an easy transfinite induction, we can build $(x^{**}_\alpha)_{\alpha<\omega_1}$ in $S_{X^{**}}$ such that
$$\forall \alpha<\omega_1,\ d\big(x^{**}_\alpha, \overline{sp}(X \cup \{x^{**}_\beta,\ \beta<\alpha\}\big)>\frac34.$$

Fix now $\alpha<\omega_1$. Since $\ell_1\not\subset X$  it follows from a result by Odell and Rosenthal \cite[Equivalences (1)-(5) on page 376]{OdellRosenthal1975}
that for each $\alpha<\omega_1$ there is a sequence $(x_{\alpha,n})_{n=1}^\infty$ in $S_X$, which converges weak$^*$ in $X^{**}$  to $x^{**}_\alpha$. In particular, the sequence $(x_{\alpha,n})_{n=1}^\infty$ is weakly Cauchy. Since $d(x^{**}_\alpha,X)>\frac34$, we may as well assume, after extracting a subsequence, that $\|x_{\alpha,n}-x_{\alpha,m}\|>\frac34$, for all $n\neq m$.\\
After passing to a further subsequence we can also assume that $(x_{\alpha,n})_{n=1}^\infty$ has a spreading model. But this means, that all the sequences of the  form $(x_{\alpha,n_{2j}}- x_{\alpha,n_{2j-1}})_{j=1}^\infty \subset 2B_X $, with $(n_j)_{j=1}^\infty$ increasing sequence in $\N$, have the same spreading model $(e_j^{\alpha})_{j=1}^\infty$. Define now $\lambda_k^{\alpha}=\|\sum_{j=1}^k e_j^{\alpha}\|.$ Since spreading models generated by weakly null sequences are $1$-suppression unconditional we have that for any $\alpha <\omega_1$, the sequence $(\lambda_k^{\alpha})_k$ is non decreasing. It also follows from our assumptions on $X$  and from the fact that $(x_{\alpha,m}-x_{\alpha,n})_{n\neq m}$ is semi-normalized that
$$\forall \alpha<\omega_1,\ \ \lim_{k \to \infty}\frac{k}{\lambda_k^{\alpha}}=\infty.$$
For fixed $\alpha<\omega_1$ and $k\in \N$, we can apply Ramsey's Theorem and, after passing to a further subsequence,  we can assume that for all
 $\mb,\nb\in [\N]^k$, with $k\le m_1 <n_1<m_2<n_2<\ldots <m_k<n_k$, we have
$$\Big\| \sum_{j=1}^k x_{\alpha, n_j}- x_{\alpha ,m_j}\Big\|\le \frac32\lambda_k^{\alpha}.$$
Using then the usual diagonalization argument we can assume that for all $k\in\N$  and for all $\mb,\nb\in [\N]^k$, with $m_1 <n_1<m_2<n_2<\ldots <m_k<n_k$,
\begin{equation}\label{interlaced}
\Big\| \sum_{j=1}^k x_{\alpha, n_j}- x_{\alpha,m_j}\Big\|\le \frac32\lambda_k^{\alpha}.
\end{equation}

Then, we define for   $\alpha<\omega_1$ and $k\in\N$,
$$f^{(k)}_\alpha: [\N]^k\to X,\quad \nb\mapsto \frac2{3\lambda_k^{\alpha}} \sum_{i=1}^k x_{\alpha, n_i}.$$
It follows from the definition of the interlaced distance, equation (\ref{interlaced}) and the monotonicity of $(\lambda_k^\alpha)_k$ that $f^{(k)}_\alpha$ is 1-Lipschitz.

Consider now $\alpha < \beta \in [1,\omega_1)$. Since $\mathrm{dist}(x^{**}_\beta, sp(x^{**}_\alpha)) > 3/4$, by Hahn-Banach, there exists an $x_{\alpha,\beta}^{***}\in S_{X^{***}}$ with $x_{\alpha,\beta}^{***}(x^{**}_\alpha) = 0$ and $x_{\alpha,\beta}^{***}(x^{**}_\beta) = \mathrm{dist}(x^{**}_\beta, sp(x^{**}_\alpha)) > 3/4$. By the principle of local reflexivity (applied to the space $X^*$) there exists $x^*_{\alpha,\beta}\in S_{X^*}$ with $x^{**}_\alpha(x_{\alpha,\beta}^{*}) = 0$ and $x^{**}_\beta(x_{\alpha,\beta}^{*}) > 3/4$. it therefore follows that for any $\mathbb{M}\in[\mathbb{N}]^\omega$:

\begin{align}
\label{separation}
\sup_{\nb\in [\M]^k} &\big\| f^{(k)}_\alpha(\nb) -f^{(k)}_\beta(\nb)\big\|\\
&\ge \limsup_{n_1\in \M, n_1\to\infty }\limsup_{n_2\in \M, n_2\to\infty }\ldots \limsup_{n_k\in \M, n_k\to\infty }  x^{*}_{\alpha,\beta} \Big(\frac2{3\lambda_k^{\beta}}\sum_{i=1}^k x_{\beta,n_i}- \frac2{3\lambda_k^{\alpha}}\sum_{i=1}^k x_{\alpha,n_i}\Big) \notag \\
&\geq \frac{2}{3\lambda_k^\beta}\frac{3k}{4} = \frac{k}{2\lambda_k^{\beta}}.\notag
\end{align}
This finishes our construction of uncountably many well separated $X$-valued Lipschitz maps.

Assume now that $X$ coarsely embeds into a Banach space $Y$ such that all the iterated duals of $Y$ are separable and let $g:X \to Y$ be such a coarse embedding. Of course, we may assume that $\omega_g(1)\le 1$. Then, for any  $\alpha <\omega_1$ and $k\in \N$, we define $g_\alpha^{(k)}=g \circ f_\alpha^{(k)}$. We have that $g_\alpha^{(k)}$ is 1-Lipschitz from $([\N]^k, d^{(k)}_{\mathrm{I}})$ to $Y$. For a fixed $k\in \N$, we can therefore apply Theorem \ref{T:2} to any uncountable subfamily of $(g_\alpha^{(k)})_{\alpha<\omega_1}$. We then deduce from (\ref{separation}) that for any $k \in \N$, $\{\alpha<\omega_1,\ \rho_{g}(\frac{k}{3\lambda_k^{\alpha}})>3\}$ is countable. This implies that the set $\{\alpha<\omega_1\ \exists k\in \N,\  \rho_{g}(\frac{k}{3\lambda_k^{\alpha}})>3\}$ is also countable and since $[1,\omega_1)$ is uncountable, there exists $\alpha<\omega_1$ such that for all $k\in \N$, $\rho_{g}(\frac{k}{3\lambda_k^{\alpha}})\le 3$. This is in contradiction with the fact that for this given $\alpha<\omega_1$, $\frac{k}{3
\lambda_k^{\alpha}}\nearrow\infty, \text{ if } k\nearrow \infty$ and $\lim_{t \to \infty}\rho_g(t)=\infty$.
\end{proof}

Understanding quantitatively what is the order of the non-separable iterated dual in Theorem F is a very interesting problem. 
\begin{prob}
Assume that $X$ is $\co$, or any separable Banach space with non separable bidual $X^{**}$ and $\ell_1\not\subset X$ such that no spreading model generated by a normalized weakly null sequence is equivalent to the $\ell_1$-unit vector basis. If $X$ coarsely embeds into a Banach space $Y$, does it imply that $Y^{**}$ is non separable?
\end{prob}



\end{document}